%% file: main.tex
\documentclass[runningheads,numbook]{svjour3}                   
\smartqed  

\usepackage{amsmath}
\usepackage{mathptmx}      
\usepackage{amsfonts}
\usepackage{bm} 
\usepackage{booktabs} 
\usepackage{multirow} 
\usepackage{graphicx}
\usepackage{epstopdf}
\usepackage{mathtools} 
\usepackage{algorithmic}
\usepackage{tikz}
\usepackage{cite}
\usepackage[hidelinks]{hyperref} 
\usepackage{xurl} 
\ifpdf
  \DeclareGraphicsExtensions{.eps,.pdf,.png,.jpg}
\else
  \DeclareGraphicsExtensions{.eps}
\fi
\usepackage{amsopn}

\DeclareMathOperator{\dx}{dx}
\DeclareMathOperator{\tr}{tr}

\newcommand{\vertiii}[1]{{\left\vert\kern-0.25ex\left\vert\kern-0.25ex\left\vert #1
        \right\vert\kern-0.25ex\right\vert\kern-0.25ex\right\vert}}

\usepackage{diagbox} 

\graphicspath{{Pictures/}} 

\spnewtheorem{assumption}{Assumption}[section]{\bf}{\rm}
\journalname{BIT}

\usepackage{etoolbox}
\newcommand*{\affaddr}[1]{#1} 
\newcommand*{\affmark}[1][*]{\textsuperscript{#1}}

\begin{document}

\title{Augmented Lagrangian preconditioners for the Oseen--Frank model of nematic and cholesteric liquid crystals\thanks{This work is supported by the National University of Defense Technology, the EPSRC Centre for Doctoral Training in Partial Differential Equations [grant number EP/L015811/1], the EPSRC Centre for Doctoral Training in Industrially Focused Mathematical Modelling [grant number EP/L015803/1] in collaboration with London Computational Solutions, and by the Engineering and Physical Sciences Research Council [grant numbers EP/R029423/1 and EP/V001493/1].}}

\titlerunning{AL preconditioners for Oseen--Frank}        

\author{Jingmin Xia \affmark[1] \and
    Patrick E.~Farrell \affmark[1] \and
    Florian Wechsung \affmark[2]
}

\authorrunning{J.~Xia et al.} 

\institute{J.~Xia \at
              \email{jingmin.xia@maths.ox.ac.uk}           
           \and
           P.~E.~Farrell \at
              \email{patrick.farrell@maths.ox.ac.uk}           
           \and
           F.~Wechsung \at
              \email{wechsung@nyu.edu}            \\
               \affaddr{\affmark[1]Mathematical Institute, University of Oxford, Oxford, UK}\\
                \affaddr{\affmark[2]Courant Institute of Mathematical Sciences, New York University, New York, USA}\\
}

\date{Received: date / Accepted: date}

\maketitle

\begin{abstract}
	We propose a robust and efficient augmented Lagrangian-type preconditioner
for solving linearizations of the Oseen--Frank model arising in nematic and cholesteric liquid crystals.
    By applying the augmented Lagrangian method, the Schur complement of the director block
    can be better approximated by the weighted mass matrix of the Lagrange multiplier,
    at the cost of making the augmented director block harder to solve.
    In order to solve the augmented director block, we develop a robust multigrid algorithm which includes an additive Schwarz relaxation that captures a pointwise version of the kernel of the semi-definite term.
    Furthermore, we prove that the augmented Lagrangian term improves the discrete enforcement of the unit-length constraint.
    Numerical experiments verify the efficiency of the algorithm and its robustness with respect to problem-related
    parameters (Frank constants and cholesteric pitch) and the mesh size.
\end{abstract}

\input{introduction}
\input{schur-approx}
\input{improv}
\input{mg}

\input{numerical_results}

\section{Conclusions}\label{sec:conclusion}

The results in this paper divide into two categories: results about the
Oseen--Frank model and its discretization, and results about the augmented Lagrangian method
for solving it.
For the former, we extended the well-posedness results of
\cite{adler-2015-article} for nematic problems to the cholesteric case.
We also showed that the Schur complement of the discretized system is spectrally
equivalent to the Lagrange multiplier mass matrix. For the latter, we showed that
the AL method improves the discrete enforcement of the constraint,
and devised a parameter-robust multigrid scheme for the augmented director
block. The key point in this is to capture the kernel of the semi-definite
augmentation term in the multigrid relaxation. Numerical experiments
validate the results and indicate that the proposed scheme outperforms
existing monolithic multigrid methods.

\bibliographystyle{spmpsci}      
\bibliography{references}
\end{document}

%% file: introduction.tex
\section{Introduction}

Liquid crystals (LC), first discovered by Reinitzer in 1888 \cite{reinitzer},
are materials that can exist in an intermediate mesophase between isotropic
liquids and solid crystals:
they can flow like liquids while also possessing long-range orientational order.
Based on different ordering symmetries, Friedel \cite{friedel} proposed to classify them into three broad categories: \textit{nematic}, \textit{smectic} and \textit{cholesteric}.
The nematic phase is the simplest and most extensively studied form of LC, where the molecules locally tend to align in one preferred direction, described in this work by a director field $\mathbf{n}: \Omega \to \mathbb{R}^3$. 
In the smectic phase, the molecules exhibit orientational order but also organize themselves into well-defined layers that can slide over each other.
In the cholesteric phase, also referred as the \textit{chiral nematic} phase, the molecules are arranged in layers, each of which is rotated with a fixed angle relative to the previous one. The distance over which the layers rotate by $2\pi$ is referred to as the cholesteric pitch $q_0$. A nonzero parameter $q_0$ indicates chirality, while a zero value of $q_0$ represents a nematic phase.
Since the orientational properties of LC can be manipulated by imposing electric fields,
they are often used to control light and have formed the basis of several important technologies in the area of display devices.
Several thorough overviews on LC modeling and its history can be found in \cite{ball-2017-article, stewart-2004-book, chand-1992-book}.

There are several models describing LC, e.g., the Oseen--Frank, Ericksen and Landau--de Gennes theories.
The Oseen--Frank model \cite{frank-1958-article, oseen-1933-article} is commonly used for the equilibrium orientation of liquid crystals.
It employs a director $\mathbf{n}: \Omega \to \mathbb{R}^3$ as the state variable and minimizes a free energy functional.
By definition, the director is a unit vector denoting the average orientation of the molecules in a fluid element at a point
and headless in the sense that $\mathbf{n}$ and $-\mathbf{n}$ are indistinguishable. The free energy functional depends on
Frank constants that describe the relative energetic costs of various kinds of distortions.
We refer to \cite{ericksen-1991-article, gennes-book} for other continuum models such as the Ericksen and the Landau--de Gennes models.
In this work, we will focus on the continuum Oseen--Frank theory.
The key difficulty is that enforcing the unit-length constraint $\mathbf{n}\cdot\mathbf{n} = 1$ with a Lagrange multiplier leads to a saddle-point system,
which poses challenges because of its poor spectral properties.
Several classical techniques regarding the solution of saddle-point problems are reviewed and illustrated in \cite{benzi-2005-article}.

There are several existing works concerning preconditioners for Oseen--Frank models of nematic LC.
For the saddle-point structure of harmonic maps (arising when all Frank constants are equal), Hu et al.\ \cite{hu-2009-article} propose to use a block-diagonal preconditioner,
consisting of a symmetric and spectrally equivalent multigrid operator and a discrete Laplacian operator.
Ramage and Gartland \cite{ramage-2013-article} consider the case of an electrically coupled equal-constant nematic LC
and combine a discretize-then-optimize approach with projection onto the nullspace of the discrete
constraint to reduce the size of the linear system. The projected problem is then preconditioned with a block-diagonal
preconditioner.
Furthermore, a number of other preconditioners are discussed and analyzed in \cite{beik-benzi-2018-article, beik-benzi-2018b-article}
for double saddle-point systems arising in both potential fluid flows and electric-field coupled nematic LC.
Concerning the double saddle-point structure, a class of Uzawa-type methods, which can be interpreted as generalized Gauss--Seidel methods,
and an augmented Lagrangian technique are studied in \cite{benzi-beik-2018-article}.
It is shown that the applied augmented Lagrangian form is mesh-independent and the performance of the iteration can be improved by increasing the value of $\gamma$.
These references also apply the discretize-then-optimize approach to tackle the pointwise unit-length vector constraint.
In this paper, we will employ the optimize-then-discretize strategy and enforce the unit-length constraint on the continuous level.
As an alternative to block preconditioning strategies, monolithic multigrid methods for the nematic problem have been proposed using Vanka \cite{adler-2015b-article} and Braess--Sarazin \cite{adler-2016-article} relaxation.

There is less work on preconditioning for cholesteric LC.
A damped Newton method with LU decomposition was applied to the bifurcation analysis of cholesteric problem in \cite{emerson-2018-article} with good results,
but no discussion of preconditioners is presented.

In this paper, we propose to enforce the unit-length constraint with an augmented Lagrangian approach
to help control the Schur complement arising in the saddle-point system.
When combined with specialized multigrid schemes, augmented Lagrangian strategies can yield 
scalable, mesh-independent, and parameter-robust preconditioners.
A notable success is the development of Reynolds-robust solvers for the
two- \cite{benzi-2006-article, olshanskii-2009-article} and three-dimensional \cite{farrell-mitchell-2018-article} stationary Navier--Stokes equations.

This success motivates the investigation of whether similar ideas can underpin robust solvers in the LC case.

The main contribution of this work is the development of a robust multigrid solver for the augmented director block and an effective Schur complement approximation for the linearization of the cholesteric Oseen--Frank equations.
The robust multigrid strategy is motivated by the general theory of Sch\"oberl and Lee et al.~\cite{lee-2007-article, schoberl-1999-phd-thesis, schoberl-1999-article}.
We develop a multigrid relaxation scheme that captures an approximation to the kernel of the semi-definite augmentation term and account for this approximation in the spectral analysis.
Furthermore, a proof of the improvement of the discrete constraint is given and verified numerically.
A key difference to previous applications of these ideas in linear elasticity and the Navier--Stokes
equations is that the constraint to be imposed on the director is nonlinear.

This paper is organized as follows.
The Oseen--Frank model is reviewed in Section \ref{sec:of} and the solvability of the discretized Newton linearizations is briefly analyzed. The augmented
Lagrangian strategy for enforcing the unit-length constraint is discussed. A Picard iteration is proposed for solving the augmented nonlinear equations.
We then give a theoretical justification of the continuous and discrete augmented Lagrangian stabilizations in Section \ref{sec:schur}.
This further leads to our choice of the approximation to the Schur complement matrix arising from the Picard iteration.
In Section \ref{sec:improv}, we prove that the augmented Lagrangian strategy improves the discrete enforcement of the constraint.
A robust multigrid algorithm for the augmented top-left block is discussed in Section \ref{sec:mg} which also includes a formal spectral analysis of our preconditioner with the property of the approximate kernel.
Numerical experiments in two-dimensional domains are reported in Section \ref{sec:numerical} to verify the effectiveness and robustness of our proposed augmented Lagrangian preconditioner.
Finally, some conclusions are presented in Section \ref{sec:conclusion}.

\section{Oseen--Frank model}\label{sec:of}

Let $\Omega \subset \mathbb{R}^d, d=\{2,3\}$ be an open, bounded domain with Lipschitz boundary $\partial\Omega$
and denote $\mathbf{H}^1_g(\Omega) = \{\mathbf{v} \in H^1(\Omega; \mathbb{R}^3): \mathbf{v}|_{\partial \Omega} = \mathbf{g}\}$
with a vector field $\mathbf{g}\in H^{1/2}(\partial\Omega;\mathbb{S}^2)$.
Here, $\mathbb{S}^2$ represents the surface of the unit ball centered at the origin.
Assume that the cholesteric LC occupying the domain $\Omega$ is equipped with a rigid anchoring (Dirichlet) boundary condition $\mathbf{n}|_{\partial\Omega}=\mathbf{g}$\footnote{The following theory also applies with mixed periodic and Dirichlet boundary conditions \cite{adler-2015-article, bedford-2014-phd}, which we shall use in some numerical examples.}.
The Oseen--Frank model (cf. \cite{frank-1958-article}) considers the following minimization problem:
\begin{equation}\label{eq:energy}
    \begin{aligned}
        & \underset{\mathbf{n} \in \mathbf{H}^1_g(\Omega)}{\min}~ J(\mathbf{n}) = \int_{\Omega} W(\mathbf{n})\dx,\\
        & \text{subject to}~ \mathbf{n}\cdot \mathbf{n} = 1 \text{ a.e.},
    \end{aligned}
\end{equation}
where the Frank energy density $W(\mathbf{n})$ is of the form
\begin{equation}\label{eq:densityk4}
    \begin{aligned}
        W(\mathbf{n}) =  \frac{K_1}{2} \left(\nabla \cdot \mathbf{n}\right)^2 &+ \frac{K_2}{2} \left(\mathbf{n} \cdot (\nabla \times \mathbf{n}) + q_0\right)^2  + \frac{K_3}{2} |\mathbf{n} \times (\nabla \times \mathbf{n})|^2\\
        &+ \frac{K_2+K_4}{2}[\tr((\nabla \mathbf{n})^2) - (\nabla \cdot\mathbf{n})^2],
    \end{aligned}
\end{equation}
where $\tr(\cdot)$ denotes the trace of a matrix, $K_i \in \mathbb{R}$ $(i=1,2,3,4)$ are elastic constants (called \textit{Frank constants}) and $q_0 \ge 0$ is the preferred pitch for the
cholesteric. $K_1$, $K_2$, $K_3$, and $K_4$ are referred to as the splay, twist,
bend, and saddle-splay constants, respectively.
Note here $\nabla\mathbf{n}$ is matrix-valued and $(\nabla\mathbf{n})^2$ denotes the matrix multiplication of the matrix $\nabla\mathbf{n}$ and itself.

If $K_1=K_2=K_3=K>0$ and $K_4=0$, the energy density \eqref{eq:densityk4} reduces to the so-called \textit{equal-constant} approximation, with energy density
    \begin{equation*}
        W(\mathbf{n}) = \frac{K}{2} \left[|\nabla \mathbf{n}|^2 + 2q_0 \mathbf{n}\cdot(\nabla\times\mathbf{n})+q_0^2\right],
    \end{equation*}
    which is a useful simplification to help us gain qualitative insight into more complex situations.

\begin{remark}
    When $q_0=0$, the energy density \eqref{eq:densityk4} corresponds to the nematic case.
    Furthermore, when combined with the equal-constant approximation,
    \eqref{eq:densityk4} reduces to
    \begin{equation}\label{eq:approxdens-n}
        W(\mathbf{n}) = \frac{K}{2} |\nabla \mathbf{n}|^2.
    \end{equation}
    With this free energy density, the solution to the minimization problem \eqref{eq:energy} is unique
    and is known as the \textit{harmonic map} from a two- or three-dimensional compact manifold to $\mathbb{S}^2$ \cite{lin-1989-article}.
    Some fast numerical algorithms for \eqref{eq:approxdens-n} have been proposed and tested in \cite{hu-2009-article}.
\end{remark}

The last term (the \textit{saddle-splay} term or the \emph{null Lagrangian}) in \eqref{eq:densityk4} can be dropped as its integral reduces to a surface integral, which is essentially a constant
if applying Dirichlet boundary conditions to the model, via the divergence theorem.
For mixed periodic and Dirichlet boundary conditions considered in Section \ref{sec:pbc}, we can verify directly that this saddle-splay energy vanishes.
Hence, for simplicity, it suffices to consider the following Frank energy density
\begin{equation*}
    W(\mathbf{n}) =  \frac{K_1}{2} \left(\nabla \cdot \mathbf{n}\right)^2 + \frac{K_2}{2} \left(\mathbf{n} \cdot (\nabla \times \mathbf{n}) + q_0\right)^2  + \frac{K_3}{2} |\mathbf{n} \times (\nabla \times \mathbf{n})|^2.
\end{equation*}

In this paper, we use a more compact form of the free energy \eqref{eq:energy} as in \cite{adler-2015-article, adler-2016-article} by introducing a symmetric dimensionless tensor
\begin{displaymath}
    \mathbf{Z} = \kappa \mathbf{n}\otimes \mathbf{n} + (\mathbf{I} - \mathbf{n}\otimes \mathbf{n})  = \mathbf{I} +(\kappa-1)\mathbf{n}\otimes \mathbf{n},
\end{displaymath}
where $\kappa = K_2/ K_3$ and $\mathbf{I}$ is the second-order identity tensor.
By the classical equality
\begin{equation}\label{eq:classic}
    |\nabla \times \mathbf{n}|^2 = \left(\mathbf{n}\cdot (\nabla\times \mathbf{n})\right)^2 + |\mathbf{n}\times (\nabla \times \mathbf{n})|^2,
\end{equation}
the original energy functional $J(\mathbf{n})$ can be written as
\begin{equation}\label{eq:comp-energy}
    \begin{aligned}
        J(\mathbf{n}) =\frac{1}{2} &\left[ K_1 \langle \nabla\cdot\mathbf{n},\nabla\cdot\mathbf{n}\rangle_0
        + K_3 \langle \mathbf{Z}\hspace{0.02cm}\nabla \times \mathbf{n} , \nabla \times \mathbf{n} \rangle_0 \right.\\
                                   & \left. + 2K_2 q_0 \langle \mathbf{n}, \nabla \times \mathbf{n}\rangle_0 + K_2 \langle q_0,q_0 \rangle_0 \right].
    \end{aligned}
\end{equation}
Here and throughout this work, $\langle \cdot, \cdot\rangle_0$ denotes the inner product in $L^2(\Omega)$ with its induced norm $\|\cdot\|_0$.
It can be observed that the auxiliary tensor $\mathbf{Z}$ contributes to the nonlinearity of $J(\mathbf{n})$ in \eqref{eq:comp-energy}.

\begin{remark}
    There is another widely used simplification of the energy density \eqref{eq:densityk4}, where $q_0=0$ and $K_2=K_3=K_1+K$, $K_4=-K$ \cite{glowinski-lin-2003-article, lin-richter-2007-article}.
    In this case, \eqref{eq:densityk4} becomes
    \begin{displaymath}
        W(\mathbf{n}) = \frac{1}{2} [K_1 |\nabla \mathbf{n}|^2 + K |\nabla \times \mathbf{n}|^2],
    \end{displaymath}
    and it is expected that as $K\rightarrow \infty$, the asymptotic behavior of minimizers provides a description of the phase transition process of
    LC from the nematic to the smectic-A phases \cite{glowinski-lin-2003-article, lin-richter-2007-article, lin-tai-2014-book}.
\end{remark}

Furthermore, it is proven in \cite[Section 2.3]{adler-2015-article} that $\mathbf{Z}$ is uniformly (with respect to $\mathbf{x}\in \Omega$) symmetric positive definite (USPD)
as long as sufficient control is maintained on $\mathbf{n}\cdot\mathbf{n}-1$.
This property of $\mathbf{Z}$ plays an essential role in proving the well-posedness of the saddle-point problem in the nematic case.
We restate the result of $\mathbf{Z}$ being USPD in the following, as it is important later:
\begin{lemma}\label{lem:zuspd}\cite[Section 2.3]{adler-2015-article}
    Assume $\alpha\le |\mathbf{n}|^2 \le \beta$~$\forall \mathbf{x}\in \Omega$ with $0<\alpha\le 1\le \beta$.
    If $\kappa>1$, then $\mathbf{Z}$ is USPD on $\Omega$;
    for $0<\kappa<1$, then $\mathbf{Z}$ is USPD on $\Omega$ if $\beta<\frac{1}{1-\kappa}$.
\end{lemma}

\begin{remark}
    Notice that the regularity of $\mathbf{n}\in \mathbf{H}^1(\Omega)$ is enough for the functional $J(\mathbf{n})$ of \eqref{eq:comp-energy} to be well defined.
    In fact, $\mathbf{n}\in \mathbf{H}^1(\Omega)$ implies $\nabla\cdot\mathbf{n}$ and $\nabla \times\mathbf{n}$
    in $\mathbf{L}^2(\Omega)$.
    By \eqref{eq:classic}, $\mathbf{n}\cdot(\nabla \times \mathbf{n})\in L^2(\Omega)$.
    This ensures that the term $\langle q_0, \mathbf{n}\cdot (\nabla \times\mathbf{n})\rangle_0$ in \eqref{eq:comp-energy} is defined.
    Furthermore, Lemma \ref{lem:zuspd} gives the boundedness of $\mathbf{Z}$, which guarantees the $L^2$-regularity of the term $\mathbf{Z}\hspace{0.02cm}\nabla\times\mathbf{n}$ in \eqref{eq:comp-energy}.
\end{remark}

Naturally, the values of elastic constants and the cholesteric pitch will be an important factor in determining the minimizers.
In order to satisfy non-negativity of the energy density, i.e.,
\begin{displaymath}
    W(\mathbf{n}) \ge 0 \quad\forall \mathbf{n} \in \mathbf{H}^1_{g}(\Omega),
\end{displaymath}
we need additional assumptions on those constants. This gives rise to Ericksen's inequalities (see \cite{ball-2017-article, bedford-2014-phd} and references therein):
\begin{displaymath}
    \begin{aligned}
        &K_1,K_2,K_3 \ge 0, K_2+K_4=0 && \text{if}\ q_0\ne 0,\\
        & 2K_1\ge K_2+K_4, K_2\ge |K_4|, K_3\ge 0 && \text{if}\ q_0=0.
    \end{aligned}
\end{displaymath}
\begin{remark}
    We have included the inequalities with regard to constant $K_4$ here for generality,
    though they are not necessary in our work as we have eliminated the $K_4$-related term in the free energy.
    In this paper, we will simply consider $K_i>0$ ($i=1,2,3$) to avoid any technical issues.
\end{remark}

For the minimization problem \eqref{eq:energy} arising in (nematic or cholesteric) liquid crystals,
it has been proven in \cite[Theorem 2.1]{lin-1989-article} that there exists a solution.
\begin{theorem}\cite[Theorem 2.1]{lin-1989-article}
    Let $\Omega$ be a bounded Lipschitz domain and assume the Dirichlet boundary data $\mathbf{g}\in {H}^{1/2}(\partial\Omega;\mathbb{S}^2)$.
    If $K_1,K_2,K_3>0$, then there exists an $\mathbf{n}\in {H}_g^1(\Omega;\mathbb{S}^2)\coloneqq\{\mathbf{n}\in H^1(\Omega;\mathbb{S}^2):\mathbf{n}=\mathbf{g}\ \text{on }\partial\Omega\}$ such that
    \begin{displaymath}
        J(\mathbf{n}) = \underset{\mathbf{u}\in {H}^1_{g}(\Omega;\mathbb{S}^2)}{\mathrm{inf}}\ J(\mathbf{u}).
    \end{displaymath}
\end{theorem}

The main difficulty in solving the Oseen--Frank model \eqref{eq:energy} is the enforcement of the unit-length constraint.
There are several existing approaches to handling constraints, e.g., projection \cite{lin-tai-2014-book},
Lagrange multipliers, and penalty methods \cite[Section 12.3 \& 17]{nocedal-1999-book}.

The projection method is numerically simple but the value of the energy functional may go up and down dramatically after each projection,
making it difficult to control in the optimization procedure \cite{lin-tai-2014-book}.
A Lagrange multiplier is often used to replace constrained optimization problems with unconstrained ones,
but an important disadvantage of this approach is that it introduces another unknown (i.e., the Lagrange multiplier)
and leads to a saddle-point structure which can be difficult to solve \cite{benzi-2005-article}.
On the other hand, the penalty method has the favorable property that the resulting system has an energy decay property \cite{lin-richter-2007-article}
which may result in an easier theoretical and numerical study of the solution.
However, the penalty parameter has to be very large for the accuracy of approximating the constraints,
leading to an ill-conditioned system.
Some works based on either projection or pure penalty methods for nematic phases can be found in \cite{glowinski-lin-2003-article, lin-richter-2007-article, glowinski-1989-book} and the references therein.

Fortunately, it is possible to amend the ill-conditioning effects with large penalty parameters that are inherent in the pure penalty method by combining it with a Lagrange multiplier.
This is the \textit{augmented Lagrangian} (AL) algorithm \cite{fortin-1983-book}.
This strategy combines the advantages of both schemes:
the penalty parameter can be relatively small due to the presence of the Lagrange multiplier,
and the Schur complement of the saddle-point system is easier to solve due to the presence of the penalty term
\cite{glowinski-lin-2003-article, glowinski-1989-book,olshanskii-2002-article, benzi-2006-article, farrell-mitchell-2018-article}.

We first consider the method of Lagrange multipliers.
We then add the augmented Lagrangian term to control the Schur complement of the system.

\subsection{Lagrange multiplier and Newton linearization}\label{sec:newton}

By introducing the Lagrange multiplier $\lambda\in L^2(\Omega)$, the associated Lagrangian of the minimization problem \eqref{eq:energy} is then defined as
\begin{equation}\label{eq:lag}
    \mathcal{L}(\mathbf{n}, \lambda) =J(\mathbf{n})+ \langle \lambda, \mathbf{n}\cdot \mathbf{n} -1 \rangle_0,
\end{equation}
and its first-order optimality conditions are:
find $(\mathbf{n}, \lambda) \in \mathbf{H}^1_g(\Omega) \times L^2(\Omega)$ such that
\begin{equation}\label{eq:equil}
    \begin{aligned}
        \mathcal{L}_{\mathbf{n}}[\mathbf{v}] &= J_\mathbf{n}[\mathbf{v}]+\langle \lambda, 2\mathbf{n}\cdot\mathbf{v}\rangle_0\\
        & = K_1 \langle \nabla\cdot \mathbf{n},\nabla\cdot\mathbf{v}\rangle_0 + K_3 \langle \mathbf{Z}\nabla\times\mathbf{n}, \nabla\times \mathbf{v}\rangle_0 \\
        & \quad + (K_2-K_3)\langle \mathbf{n}\cdot\nabla\times \mathbf{n}, \mathbf{v}\cdot\nabla\times\mathbf{n}\rangle_0 \\
        & \quad + K_2q_0\langle \mathbf{v}, \nabla\times\mathbf{n}\rangle_0 + K_2q_0\langle \mathbf{n}, \nabla\times\mathbf{v}\rangle_0 + \langle \lambda,2\mathbf{n}\cdot\mathbf{v}\rangle_0 \\
        & = 0 \quad \forall \mathbf{v}\in \mathbf{H}^1_0(\Omega),\\
        \mathcal{L}_{\lambda}[\mu] &= \langle \mu, \mathbf{n}\cdot\mathbf{n}-1 \rangle_0 = 0
        \quad \forall \mu \in L^2(\Omega).
    \end{aligned}
\end{equation}

As \eqref{eq:equil} is nonlinear, Newton linearization is employed.
Let $\mathbf{n}_k$ and $\lambda_k$ be the current approximations for $\mathbf{n}$ and $\lambda$, respectively,
and denote the corresponding updates to these approximations as $\delta \mathbf{n} = \mathbf{n}_{k+1} - \mathbf{n}_k$ and $\delta \lambda = \lambda_{k+1} - \lambda_k$.
Then the Newton iteration at $(\mathbf{n}_k,\lambda_k)$ in block form is given by:
find $(\delta \mathbf{n},\delta\lambda)\in \mathbf{H}^1_0(\Omega)\times L^2(\Omega)$ such that
\begin{equation}\label{eq:newton-it}
    \begin{bmatrix}
        \mathcal{L}_{\mathbf{n}\mathbf{n}} & \mathcal{L}_{\mathbf{n}\lambda} \\
        \mathcal{L}_{\lambda\mathbf{n}} & 0 
    \end{bmatrix}
    \begin{bmatrix}
        \delta \mathbf{n} \\
        \delta \lambda 
    \end{bmatrix} = -
        \begin{bmatrix}
             \mathcal{L}_{\mathbf{n}}\\
             \mathcal{L}_{\lambda}
        \end{bmatrix},
\end{equation}
where
\begin{equation}\label{eq:lnn}
    \begin{split}
        \mathcal{L}_{\mathbf{n}\mathbf{n}}[\mathbf{v}, \delta\mathbf{n}] &=
        J_{\mathbf{n}\mathbf{n}}[\mathbf{v},\delta\mathbf{n}] + \langle \lambda_k, 2\delta\mathbf{n}\cdot\mathbf{v}\rangle_0 \\
        &= K_1 \langle \nabla\cdot\delta\mathbf{n}, \nabla\cdot\mathbf{v}\rangle_0 + K_3\langle \mathbf{Z}(\mathbf{n}_k) \nabla\times\delta\mathbf{n}, \nabla\times\mathbf{v}\rangle_0 \\
        &\quad + (K_2-K_3)\Big( \langle \delta\mathbf{n}\cdot\nabla\times\mathbf{n}_k, \mathbf{n}_k\cdot\nabla\times\mathbf{v}\rangle_0 + \langle \mathbf{n}_k\cdot\nabla\times\mathbf{n}_k, \delta\mathbf{n}\cdot\nabla\times\mathbf{v}\rangle_0 \\
        &\quad + \langle \mathbf{v}\cdot\nabla\times\mathbf{n}_k,\mathbf{n}_k\cdot\nabla\times\delta\mathbf{n}\rangle_0 + \langle \mathbf{n}_k\cdot\nabla\times\mathbf{n}_k,\mathbf{v}\cdot\nabla\times\delta\mathbf{n}\rangle_0 \\
        &\quad + \langle \delta \mathbf{n}\cdot\nabla\times\mathbf{n}_k, \mathbf{v}\cdot\nabla\times\mathbf{n}_k\rangle_0 \Big) \\
        &\quad + K_2q_0\langle \mathbf{v}, \nabla\times\delta\mathbf{n}\rangle_0 + K_2q_0\langle \delta\mathbf{n}, \nabla\times\mathbf{v}\rangle_0 + \langle \lambda_k,2\delta\mathbf{n}\cdot\mathbf{v}\rangle_0,
    \end{split}
\end{equation}
and
\begin{displaymath}
    \begin{aligned}
        &  \mathcal{L}_{\mathbf{n}\lambda}[\mathbf{v}, \delta\lambda] = \langle \delta\lambda, 2\mathbf{n}_k\cdot\mathbf{v}\rangle_0,\\
        &  \mathcal{L}_{\lambda\mathbf{n}}[\mu, \delta\mathbf{n}] = \langle \mu, 2\mathbf{n}_k\cdot\delta\mathbf{n}\rangle_0.
    \end{aligned}
\end{displaymath}

Since $\mathcal{L}(\mathbf{n}, \lambda)$ is linear in $\lambda$, $\mathcal{L}_{\lambda\lambda} = 0$.
This results in \eqref{eq:newton-it} being a saddle-point problem.

With a suitable spatial discretization (we only consider conforming finite elements in this work, i.e., $V_h\subset \mathbf{H}^1_0(\Omega)$, $Q_h\subset L^2(\Omega)$),
a symmetric saddle-point system must be solved at each Newton iteration:
\begin{equation}\label{eq:saddle}
    \begin{bmatrix}
        A & B^\top \\
        B & 0
    \end{bmatrix}
    \begin{bmatrix}
        U \\
        P
    \end{bmatrix}=
    \begin{bmatrix}
        f \\
        g
    \end{bmatrix},
\end{equation}
where
$U$ and $P$ represent the coefficient vectors of $\delta \mathbf{n}$ and $\delta \lambda$ in terms of the basis functions of $V_h$ and $Q_h$, respectively.

We can accordingly write the discrete variational problem as:
find $\delta\mathbf{n}_h\in V_h$ and $\delta\lambda_h\in Q_h$ such that
\begin{equation}\label{eq:bilinear}
    \begin{aligned}
        a(\delta \mathbf{n}_h, \mathbf{v}_h) + b(\mathbf{v}_h, \delta \lambda_h) & = F(\mathbf{v}_h)
        \quad \forall \mathbf{v}_h\in V_h, \\
        b(\delta \mathbf{n}_h, \mu_h) &= G(\mu_h)
        \quad \forall \mu_h \in Q_h,
    \end{aligned}
\end{equation}
where $a(\cdot, \cdot)$ and $b(\cdot,\cdot)$ are bilinear forms given by
\begin{equation*}
    \begin{split}
        a(\mathbf{u},\mathbf{v}) = &K_1 \langle \nabla\cdot\mathbf{u}, \nabla\cdot\mathbf{v}\rangle_0 + K_3\langle \mathbf{Z}(\mathbf{n}_k) \nabla\times\mathbf{u}, \nabla\times\mathbf{v}\rangle_0 \\
        &+ (K_2-K_3)\Big( \langle \mathbf{u}\cdot\nabla\times\mathbf{n}_k, \mathbf{n}_k \cdot\nabla\times\mathbf{v}\rangle_0 + \langle \mathbf{n}_k \cdot\nabla\times\mathbf{n}_k, \mathbf{u}\cdot\nabla\times\mathbf{v}\rangle_0 \\
        &+ \langle \mathbf{v}\cdot\nabla\times\mathbf{n}_k,\mathbf{n}_k\cdot\nabla\times\mathbf{u}\rangle_0 + \langle \mathbf{n}_k\cdot\nabla\times\mathbf{n}_k, \mathbf{v}\cdot\nabla\times\mathbf{u}\rangle_0 \\
        &+ \langle\mathbf{u}\cdot\nabla\times\mathbf{n}_k, \mathbf{v}\cdot\nabla\times\mathbf{n}_k\rangle_0 \Big) \\
        &+ K_2q_0\langle\mathbf{v}, \nabla\times\mathbf{u}\rangle_0 + K_2q_0\langle \mathbf{u}, \nabla\times\mathbf{v}\rangle_0 + \langle \lambda_k,2\mathbf{u}\cdot\mathbf{v}\rangle_0,
    \end{split}
\end{equation*}
and
\begin{displaymath}
    b(\mathbf{v},p) = \langle p, 2\mathbf{n}_k\cdot\mathbf{v}\rangle_0,
\end{displaymath}
and $F$ and $G$ are linear functionals in the forms of
\begin{equation*}
    \begin{aligned}
    F(\mathbf{v}) = &- \Big( K_1 \langle \nabla\cdot \mathbf{n}_k, \nabla\cdot \mathbf{v}\rangle_0 + K_3\langle Z(\mathbf{n}_k)\nabla\times \mathbf{n}_k, \nabla\times\mathbf{v}\rangle_0 \\
    &+ (K_2-K_3)\langle\mathbf{n}_k\cdot\nabla\times\mathbf{n}_k, \mathbf{v}\cdot\nabla\cdot\mathbf{n}_k\rangle_0 \\
        &+ K_2q_0 \langle \mathbf{v}, \nabla\times\mathbf{n}_k \rangle_0 + K_2q_0 \langle \mathbf{n}_k,\nabla\times\mathbf{v}\rangle_0 \\
        &+ \langle \lambda_k, 2\mathbf{n}_k\cdot\mathbf{v}\rangle_0 \Big),
    \end{aligned}
\end{equation*}
and
\begin{displaymath}
    G(\mu) = -\langle \mu, \mathbf{n}_k\cdot\mathbf{n}_k -1\rangle_0.
\end{displaymath}

\begin{remark}
    The well-posedness of the continuous and discretized Newton system (with the $([\mathbb{Q}_m]^d\oplus B_F)$-$\mathbb{Q}_0$ finite element pair, $m\ge 1$)
    for a generalized nematic LC problem is discussed in \cite{adler-2015-article},
    where $B_F$ denotes the space of quadratic bubbles and $\mathbb{Q}_k$ represents tensor product piecewise $C^0$ polynomials of degree $k \ge 0$ on a quadrilateral mesh.
    Moreover, the authors of \cite{adler-2016-article} considered the pure penalty approach for nematic LC
    and obtained a well-posedness result of the penalized Newton iteration through similar techniques.
    We will follow these analysis strategies in this section.
\end{remark}

In our work, we will denote by $\mathbb{P}_k$ the set of piecewise $C^0$ polynomials of degree $k\ge 0$ on a mesh of triangles or tetrahedra.

It is straightforward to deduce the well-posedness of the discrete Newton iteration \eqref{eq:bilinear}
for cholesteric problems under some proper assumptions on the problem-dependent constants.
In fact, two additional $q_0$-related terms in $\mathcal{L}_{\mathbf{n}\mathbf{n}}$ from \eqref{eq:lnn}
compared to the nematic energy density from \cite{adler-2015-article} are simply $L^2$ inner products,
which can be easily bounded using the Cauchy--Schwarz and triangle inequalities.
We state the results without proof in the following and start with some assumptions.

\begin{assumption}\label{assump}
    Assume that there exist constants $0<\alpha\le 1\le \beta$ such that $\alpha \le |\mathbf{n}_k|^2 \le \beta$.
    For $0<\kappa<1$, assume further that $\beta < \frac{1}{1-\kappa}$.
    By Lemma \ref{lem:zuspd}, $\mathbf{Z}(\mathbf{n}_k)$ remains USPD with lower bound $\eta$ and upper bound $\Lambda$, i.e.,
    \begin{displaymath}
        \eta \le \frac{\mathbf{x}^\top \mathbf{Z}(\mathbf{n}_k)\mathbf{x}}{\mathbf{x}^\top \mathbf{x}} \le \Lambda
        \quad \forall \mathbf{x}\in\mathbb{R}^d\backslash\{\mathbf{0}\}.
    \end{displaymath}
\end{assumption}

Note that here and hereafter, $\|\cdot\|_1$ denotes the $H^1$ norm: $\|w\|_1^2=\|w\|_0^2+\|\nabla w\|_0^2$.

\begin{lemma}\label{lem:acoer:dis}
    (Continuous coercivity)
    With Assumption \ref{assump}, we assume further that the current Lagrange multiplier approximation $\lambda_k$ is pointwise non-negative almost everywhere.
    Let $K_1>K_2q_0C_4$ and $K_3\eta > K_2q_0(C_4 + 1)$ with $C_4$ to be defined.
    Then there exists an $\alpha_0>0$ such that
    \begin{equation}\label{eq:coer}
        \alpha_0 \|\mathbf{v}\|_1^2 \le a(\mathbf{v},\mathbf{v})
        \quad \forall\mathbf{v}\in \mathbf{H}^1_0(\Omega).
    \end{equation}
    Moreover, when $\kappa=1$, i.e., $K_2=K_3$, if $K_1>K_2q_0C_4$ and $1>q_0(C_4 + 1)$,
    then the coercivity result \eqref{eq:coer} also holds.
\end{lemma}
\begin{proof}
    With the lower bound $\eta$ of $\mathbf{Z}$, we compute the bilinear form:
    \begin{equation*}
        \begin{split}
            a(\mathbf{v},\mathbf{v}) &\ge K_1\|\nabla\cdot\mathbf{v}\|_0^2 + K_3\eta \|\nabla\times\mathbf{v}\|_0^2 + 2K_2q_0 \langle \mathbf{v},\nabla\times\mathbf{v}\rangle_0 + 2\langle \lambda_k, \mathbf{v}\cdot\mathbf{v}\rangle_0\\
            & \ge K_1\|\nabla\cdot\mathbf{v}\|_0^2 + K_3\eta \|\nabla\times\mathbf{v}\|_0^2 - 2K_2q_0 |\langle \mathbf{v},\nabla\times\mathbf{v}\rangle_0| \\
            & \ge K_1\|\nabla\cdot\mathbf{v}\|_0^2 + K_3\eta \|\nabla\times\mathbf{v}\|_0^2 - 2K_2q_0\|\mathbf{v}\|_0\|\nabla\times\mathbf{v}\|_0 \\
            & \ge K_1\|\nabla\cdot\mathbf{v}\|_0^2 + K_3\eta \|\nabla\times\mathbf{v}\|_0^2 - K_2q_0(\|\mathbf{v}\|_0^2 + \|\nabla\times\mathbf{v}\|_0^2 ),
        \end{split}
    \end{equation*}
    where the first inequality comes from the assumption that $\lambda_k$ is non-negative pointwise
    and the last two inequalities are derived by Cauchy--Schwarz and H\"older inequalities, respectively.

    By Remark 2.7 of \cite{girault-2011-book}, for a bounded Lipschitz domain, there exists $C_1>0$ such that
    \begin{equation*}
        \|\nabla\mathbf{v}\|_0^2 \le C_1(\|\nabla\cdot\mathbf{v}\|^2_0 + \|\nabla\times \mathbf{v}\|_0^2),
    \end{equation*}
    for all $\mathbf{v}\in \mathbf{H}_0(\mathrm{div},\Omega)\cap\mathbf{H}_0(\mathrm{curl},\Omega)$\footnote{In fact, $\mathbf{H}_0^1(\Omega)=\mathbf{H}_0(\mathrm{div},\Omega)\cap\mathbf{H}_0(\mathrm{curl},\Omega)$ holds for any bounded Lipschitz domain $\Omega$ \cite[Lemma 2.5]{girault-2011-book}.}.
    Here, we denote
    \begin{displaymath}
        \begin{aligned}
            &\mathbf{H}_0(\mathrm{div},\Omega)=\{\mathbf{v}\in \mathbf{L}^2(\Omega):\nabla\cdot\mathbf{v}\in L^2(\Omega), \bm{\nu}\cdot\mathbf{v}=0 \ \text{on}\ \partial\Omega\}, \\
            &\mathbf{H}_0(\mathrm{curl},\Omega)=\{\mathbf{v}\in \mathbf{L}^2(\Omega):\nabla\times\mathbf{v}\in \mathbf{L}^2(\Omega), \bm{\nu}\times\mathbf{v}=\mathbf{0} \ \text{on}\ \partial\Omega\},
        \end{aligned}
    \end{displaymath}
    where $\bm{\nu}$ is the outward unit normal on the boundary $\partial\Omega$.
    Then using the classical Poincar\'e inequality, $\|\mathbf{v}\|_0^2\le C_3\|\nabla\mathbf{v}\|_0^2$ for all $\mathbf{v}\in \mathbf{H}^1_0(\Omega)$,
    and defining $C_4=C_1C_3>0$, we have
    \begin{equation*}
        \|\mathbf{v}\|_0^2 \le C_4(\|\nabla\cdot\mathbf{v}\|_0^2 + \|\nabla\times \mathbf{v}\|_0^2).
    \end{equation*}
    Furthermore, there exists $C_2=C_4+C_1>0$ such that
    \begin{equation*}
        \|\mathbf{v}\|_1^2 \le C_2 (\|\nabla\cdot\mathbf{v}\|_0^2 + \|\nabla\times\mathbf{v}\|_0^2).
    \end{equation*}
    It follows that
    \begin{equation*}
        \begin{split}
            a(\mathbf{v},\mathbf{v}) & \ge K_1\|\nabla\cdot\mathbf{v}\|_0^2
            + K_2 \|\nabla\times\mathbf{v}\|_0^2
            - K_2q_0\left[C_4 \left(\|\nabla\cdot\mathbf{v}\|_0^2 + \|\nabla\times \mathbf{v}\|_0^2\right)
            - \|\nabla\times\mathbf{v}\|^2_0\right] \\
            & = (K_1-K_2q_0 C_4)\|\nabla\cdot\mathbf{v}\|_0^2
            + (K_3\eta - K_2q_0C_4-K_2q_0 )\|\nabla\times\mathbf{v}\|_0^2.
        \end{split}
    \end{equation*}
    Choosing $c=\min\{K_1-K_2q_0C_4, K_3\eta-K_2q_0C_4-K_2q_0 \}>0$ (the positivity follows from the assumptions) and $\alpha_0=c/C_2$,
    we find that the coercivity \eqref{eq:coer} holds.

    In particular, when $\kappa=1$ (i.e., $K_2=K_3$), we have $\mathbf{Z}=\mathbf{I}$ and thus $\eta=1$.
    Then, the bilinear form becomes
    \begin{equation*}
        \begin{split}
            a(\mathbf{v},\mathbf{v}) &= K_1\|\nabla\cdot\mathbf{v}\|_0^2 + K_2 \|\nabla\times\mathbf{v}\|_0^2 + 2K_2q_0 \langle \mathbf{v},\nabla\times\mathbf{v}\rangle_0 + 2\langle \lambda_k, \mathbf{v}\cdot\mathbf{v}\rangle_0\\
            & \ge K_1\|\nabla\cdot\mathbf{v}\|_0^2 + K_2 \|\nabla\times\mathbf{v}\|_0^2 - 2K_2q_0 |\langle \mathbf{v},\nabla\times\mathbf{v}\rangle_0| \\
            & \ge K_1\|\nabla\cdot\mathbf{v}\|_0^2 + K_2 \|\nabla\times\mathbf{v}\|_0^2 - 2K_2q_0\|\mathbf{v}\|_0\|\nabla\times\mathbf{v}\|_0 \\
            & \ge K_1\|\nabla\cdot\mathbf{v}\|_0^2 + K_2 \|\nabla\times\mathbf{v}\|_0^2 - K_2q_0(\|\mathbf{v}\|_0^2 + \|\nabla\times\mathbf{v}\|_0^2 ).
        \end{split}
    \end{equation*}
    By choosing $C = \min\{K_1-K_2q_0 C_4, K_2(1 - q_0 C_4 - q_0)\}>0$ (the positivity comes from the assumptions)
    and $\alpha_0=C/C_2$, we obtain the desired coercivity
    \begin{displaymath}
        a(\mathbf{v},\mathbf{v})\ge \alpha_0\|\mathbf{v}\|_1^2
        \quad \forall \mathbf{v}\in \mathbf{H}^1_0(\Omega),
    \end{displaymath}
    as stated in \eqref{eq:coer}. \qed
\end{proof}

So far, the coercivity of the bilinear form $a(\cdot,\cdot)$ has been shown for all functions in $\mathbf{H}^1_0(\Omega)$.
The discrete coercivity follows if a conforming finite element for the director space is chosen.

The boundedness of the bilinear form $a(\cdot,\cdot)$ and the right hand side functionals $F(\cdot)$ and $G(\cdot)$
can be obtained directly by following the proofs in~\cite{adler-2015-article}.
Hence, we omit the details here.

It remains to consider the discrete inf-sup condition of the bilinear form $b(\cdot,\cdot)$
for a finite element pair $V_h$-$Q_h$, i.e.~whether there exists a constant $c$ such that
\begin{equation*}
    \underset{\mathbf{u}_h\in V_h\backslash \{0\}}{\mathrm{sup}} \frac{b(\mathbf{u}_h, \mu_h)}{\|\mathbf{u}_h\|} \ge c\|\mu_h\|
    \quad \forall \mu_h \in Q_h.
\end{equation*}
The continuous inf-sup condition was shown in \cite[Appendix B]{emerson-phd} and \cite[Theorem 3.1]{hu-2009-article}.
%
%
However, the discrete inf-sup condition is not inherited from the continuous problem.
Some previous works have succeeded in obtaining a discrete inf-sup condition for some specific discretizations.
A discrete inf-sup condition was proven for the $([\mathbb{Q}_m]^d\oplus B_F)$-$\mathbb{Q}_0$ element on quadrilaterals
in \cite[Lemma 2.5.14]{emerson-phd} and \cite[Lemma 3.12]{adler-2015-article}.
The discrete inf-sup condition for the $[\mathbb{P}_1]^2$-$\mathbb{P}_1$ discretization is shown in \cite[Theorem 4.5]{hu-2009-article},
where the analysis is only valid for the two-dimensional case due to the use of some special inverse inequalities.
It is straightforward to deduce that an enrichment of $V_h$ still guarantees the stability of the discretization, and thus $[\mathbb{P}_2]^2$-$\mathbb{P}_1$ is inf-sup stable under
the same conditions.
%


We now consider the matrix form of the saddle-point system \eqref{eq:saddle}.
The coercivity of the bilinear form $a(\cdot,\cdot)$ implies the invertibility of the coefficient matrix $A$
and the discrete inf-sup condition indicates that $B$ has full row rank.
We use the full block factorization preconditioner
\begin{equation*}
\mathcal{P}^{-1} =
    \begin{bmatrix}
        I & -\tilde{A}^{-1}B^\top\\
        0 & I
    \end{bmatrix}
    \begin{bmatrix}
        \tilde{A}^{-1} & 0\\
        0 & \tilde{S}^{-1}
    \end{bmatrix}
    \begin{bmatrix}
        I & 0\\
        -B\tilde{A}^{-1} & I
    \end{bmatrix}
\end{equation*}
with approximate inner solves $\tilde{A}^{-1}$ and $\tilde{S}^{-1}$ for the director block and the Schur complement $S = - BA^{-1}B^\top$, respectively,
for solving the saddle-point problem \eqref{eq:saddle}. With exact inner solves, this is an exact inverse.
With this strategy, solving the original saddle-point problem \eqref{eq:saddle} reduces to solving two smaller linear systems involving $A$ and $S$.
Even though $A$ is sparse, its inverse is generally dense, making it impractical to store $S$ explicitly.
In this situation, developing a fast solver for $A$ is tractable while approximating $S$ becomes difficult.
We will return to this issue in Section \ref{sec:schur} and Section \ref{sec:mg}.

\subsection{Augmented Lagrangian form}

Now, we employ the AL stabilization strategy and modify the linearized saddle point system to control its Schur complement $S$.

\subsubsection{Penalizing the constraint}\label{sec:cont-penal}

We penalize the continuous form of the nonlinear constraint $\mathbf{n}\cdot\mathbf{n}=1$ in the AL algorithm and obtain the Lagrangian
\begin{equation}\label{eq:aug-lag}
    \tilde{\mathcal{L}}(\mathbf{n}, \lambda) = \mathcal{L}(\mathbf{n}, \lambda) + \frac{\gamma}{2} \langle \mathbf{n}\cdot\mathbf{n} -1, \mathbf{n}\cdot\mathbf{n} -1 \rangle_0
\end{equation}
for $\gamma \ge 0$.
The weak form of the associated first-order optimality conditions is to
find $(\mathbf{n}, \lambda)\in \mathbf{H}^1_g(\Omega)\times L^2(\Omega)$ such that
\begin{equation*}
    \begin{aligned}
        \tilde{\mathcal{L}}_{\mathbf{n}}[\mathbf{v}] &= \mathcal{L}_{\mathbf{n}}[\mathbf{v}] + 2\gamma \langle \mathbf{n}\cdot\mathbf{n}-1, \mathbf{n}\cdot\mathbf{v} \rangle_0 = 0
        \quad \forall \mathbf{v}\in \mathbf{H}^1_0(\Omega),\\
        \tilde{\mathcal{L}}_{\lambda}[\mu] &= \mathcal{L}_{\lambda}[\mu]=\langle \mu, \mathbf{n}\cdot\mathbf{n}-1\rangle_0 = 0
        \quad \forall \mu \in L^2(\Omega).
    \end{aligned}
\end{equation*}
The Newton linearization at a given approximation $(\mathbf{n}_k,\lambda_k)$ yields a system of the form:
\begin{displaymath}
    \begin{bmatrix}
        \tilde{\mathcal{L}}_{\mathbf{n}\mathbf{n}} & {\mathcal{L}}_{\mathbf{n}\lambda}\\
        {\mathcal{L}}_{\lambda\mathbf{n}} & 0
    \end{bmatrix}
    \begin{bmatrix}
        \delta\mathbf{n}\\
        \delta \lambda
    \end{bmatrix}=-
    \begin{bmatrix}
        \tilde{\mathcal{L}}_\mathbf{n}\\
        {\mathcal{L}}_\lambda
    \end{bmatrix}.
\end{displaymath}
Thus, we have to solve the augmented discrete variational problem:
\begin{equation}\label{eq:newton-aug}
    \begin{aligned}
        {a}^c(\delta\mathbf{n}_h, \mathbf{v}_h) + b(\mathbf{v}_h, \delta\lambda_h) & = {F}^c(\mathbf{v}_h)
        \quad \forall \mathbf{v}_h\in V_h, \\
        b(\delta\mathbf{n}_h, \mu_h) &= G(\mu_h)
        \quad \forall \mu_h \in Q_h,
    \end{aligned}
\end{equation}
where
\begin{equation*}
    {a}^c(\mathbf{u},\mathbf{v}) = a(\mathbf{u},\mathbf{v}) + 4\gamma\langle \mathbf{n}_k\cdot\mathbf{u}, \mathbf{n}_k\cdot \mathbf{v}\rangle_0 + 2\gamma\langle \mathbf{n}_k\cdot\mathbf{n}_k -1, \mathbf{u}\cdot\mathbf{v}\rangle_0,
\end{equation*}
and
\begin{equation*}
    {F}^c(\mathbf{v}) = F(\mathbf{v}) - 2\gamma\langle \mathbf{n}_k\cdot\mathbf{n}_k-1, \mathbf{n}_k\cdot\mathbf{v}\rangle_0.
\end{equation*}

Comparing \eqref{eq:newton-aug} to the original system \eqref{eq:bilinear},
only the bilinear form $a(\cdot,\cdot)$ and the right hand side functional $F(\cdot)$ have changed.
The boundedness of $F^c(\cdot)$ follows straightforwardly via the Cauchy--Schwarz inequality.
As for the coercivity of $a^c(\cdot,\cdot)$, an additional assumption on the penalty parameter $\gamma$ is needed.

\begin{lemma}\label{lem:acoer-aug}
    (Continuous coercivity)
    Let $\alpha_0 > 0$ be the coercivity constant of $a(\cdot, \cdot)$.
    If $\alpha_0>2\gamma|\alpha-1|$ with $0<\alpha\le 1\le\beta$ satisfying $\alpha\le |\mathbf{n}_k|^2\le \beta$,
there exists a $\beta_0>0$ such that
    \begin{equation*}
       a^c(\mathbf{v},\mathbf{v})\ge \beta_0 \|\mathbf{v}\|_1^2
        \quad \forall\mathbf{v}\in \mathbf{H}^1_0(\Omega).
    \end{equation*}
\end{lemma}
\begin{proof}
    Note that
    \begin{displaymath}
        \begin{aligned}
            a^c(\mathbf{v},\mathbf{v}) &= a(\mathbf{v},\mathbf{v})+4\gamma\|\mathbf{n}_k\cdot\mathbf{v}\|_0^2
        +2\gamma\langle \mathbf{n}_k\cdot\mathbf{n}_k-1, \mathbf{v}\cdot\mathbf{v}\rangle_0\\
            &\ge a(\mathbf{v},\mathbf{v})
        +2\gamma\langle \mathbf{n}_k\cdot\mathbf{n}_k-1, \mathbf{v}\cdot\mathbf{v}\rangle_0.
        \end{aligned}
    \end{displaymath}
    By the assumption that $a(\mathbf{v},\mathbf{v})\ge \alpha_0\|\mathbf{v}\|_1^2$ for some $\alpha_0>0$,
    we have
    \begin{displaymath}
            a^c(\mathbf{v},\mathbf{v})
            \ge \alpha_0 \|\mathbf{v}\|_1^2
        +2\gamma\langle \mathbf{n}_k\cdot\mathbf{n}_k-1, \mathbf{v}\cdot\mathbf{v}\rangle_0.
    \end{displaymath}
    Moreover, since $\mathbf{n}_k\cdot\mathbf{n}_k\ge \alpha$ and $\alpha- 1\le 0$, we get
    \begin{displaymath}
            2\gamma\langle \mathbf{n}_k\cdot\mathbf{n}_k-1, \mathbf{v}\cdot\mathbf{v}\rangle_0 \ge
        2\gamma(\alpha-1)\|\mathbf{v}\|_0^2
            \ge 2\gamma(\alpha-1)\|\mathbf{v}\|_1^2.
    \end{displaymath}
    Thus, by taking $\beta_0=\alpha_0-2\gamma|\alpha-1|>0$, we obtain the desired coercivity property. \qed
\end{proof}

The condition $\alpha_0 > 2\gamma |\alpha-1|$ in Lemma \ref{lem:acoer-aug} indicates a limit on the value of $\gamma$ to ensure the solvability of the augmented system \eqref{eq:newton-aug}.
However, it is desirable to use large values of $\gamma$ to achieve better control of the Schur complement. We therefore choose to employ a Picard
iteration to solve the nonlinear problem, omitting the term $2\gamma\langle \mathbf{n}_k\cdot\mathbf{n}_k-1, \mathbf{v}\cdot\mathbf{v}\rangle_0$
from the linearized equations. This yields the linearized problem: find $(\delta\mathbf{n}_h, \delta\lambda_h) \in V_h \times Q_h$ such that
\begin{equation}\label{eq:modified-newton-aug}
    \begin{aligned}
        {a}^{m}(\delta\mathbf{n}_h, \mathbf{v}_h) + b(\mathbf{v}_h, \delta\lambda_h) & = {F}^c(\mathbf{v}_h)
        \quad \forall \mathbf{v}_h\in V_h, \\
        b(\delta\mathbf{n}_h, \mu_h) &= G(\mu_h)
        \quad \forall \mu_h \in Q_h,
    \end{aligned}
\end{equation}
with the modified bilinear form
\begin{equation}
    {a}^m(\mathbf{u},\mathbf{v}) = a(\mathbf{u},\mathbf{v}) + 4\gamma\langle \mathbf{n}_k\cdot\mathbf{u}, \mathbf{n}_k\cdot \mathbf{v}\rangle_0
\end{equation}
to be solved at each nonlinear iteration.
This ensures that the $(1, 1)$-block is coercive with a coercivity constant independent of $\gamma$.
Moreover, in contrast to the situation with the Navier--Stokes equations, numerical experiments indicate that the use of this Picard
requires \emph{fewer} nonlinear iterations to converge to a given tolerance than using the full Newton linearization (see Section \ref{sec:pbc}).

The corresponding matrix form of the variational problem \eqref{eq:newton-aug} becomes
\begin{equation}\label{eq:saddle-aug}
    \begin{bmatrix}
        A + \gamma A_* & B^\top\\
        B & 0
    \end{bmatrix}
    \begin{bmatrix}
        U\\
        P
    \end{bmatrix}=
    \begin{bmatrix}
        f+ \gamma l\\
        g
    \end{bmatrix},
\end{equation}
where $A_*$ is the assembly of
$4\langle \mathbf{n}_k\cdot\mathbf{u}, \mathbf{n}_k\cdot \mathbf{v}\rangle_0$
and $l$ denotes the assembly of $- 2\langle \mathbf{n}_k\cdot\mathbf{n}_k-1, \mathbf{n}_k\cdot\mathbf{v}\rangle_0$.
Note that compared to the non-augmented version \eqref{eq:saddle}, the $(1,1)$ block in \eqref{eq:saddle-aug} has an additional semi-definite term $\gamma A_*$ with a large coefficient $\gamma$.
Its sparsity pattern remains unchanged.
We will construct a robust multigrid method to solve this top-left block in Section \ref{sec:mg}.

Since the unit-length constraint is enforced exactly in \eqref{eq:aug-lag},
the continuous solutions to minimizing both \eqref{eq:aug-lag} and \eqref{eq:lag} are the same.
However, the unit-length constraint is not enforced exactly in our finite element
discretization, and hence this stabilization does change the computed discrete solution.

\begin{remark}
    When applying the augmented Lagrangian strategy, one can apply it before discretization or afterwards.
    In this work we apply the continuous penalization, as it improves the enforcement of the nonlinear
    constraint, as shown later in Section \ref{sec:improv}.
    This is different to the approach considered in \cite{benzi-2006-article, farrell-mitchell-2018-article}
    for the stationary Navier--Stokes equations, where the discrete AL stabilization was used to yield a
    system that has the same solution but a better Schur complement.
\end{remark}

%% file: schur-approx.tex
\section{Approximation to the Schur complement}\label{sec:schur}

The Schur complement of the augmented director block in \eqref{eq:saddle-aug} is given by
\begin{equation*}
    {S}_\gamma = -B A_\gamma^{-1} B^\top = -B(A+\gamma A_*)^{-1}B^\top.
\end{equation*}
We now proceed to analyze this Schur complement by following similar techniques to those of \cite[\S 4]{heister-2012-article}. We will show that $A_*$ is equal to the matrix arising from the \emph{discrete} AL stabilization (which controls the Schur complement) plus a perturbation that vanishes as the mesh is refined.

Let $\Pi_{Q_h}: Q\rightarrow Q_h$ be the orthogonal $L^2$ projection operator, i.e.,
\begin{equation*}
    \langle p-\Pi_{Q_h}p, q\rangle_0 =0 \quad\forall q\in Q_h.
\end{equation*}
We define the fluctuation operator $\kappa\coloneqq {I}-\Pi_{Q_h}$ where ${I}:Q\rightarrow Q$ is the identity mapping.
Therefore, one has
\begin{equation*}
    \langle \kappa(p), q\rangle_0 =0 \quad\forall q\in Q_h.
\end{equation*}
For $\mathbf{u}_h,\mathbf{v}_h\in V_h$, one can split the term $4\langle \mathbf{n}_k \cdot \mathbf{u}_h, \mathbf{n}_k \cdot \mathbf{v} \rangle_0$ into the following terms using the properties of $\kappa$ and $\Pi_{Q_h}$:
\begin{equation*}
    \begin{aligned}
        4\langle \mathbf{n}_k \cdot \mathbf{u}, \mathbf{n}_k \cdot \mathbf{v} \rangle_0
    &= \langle \Pi_{Q_h}(2 \mathbf{n}_k\cdot\mathbf{n}), 2\mathbf{n}_k\cdot\mathbf{v}\rangle_0 + \langle \kappa(2\mathbf{n}_k\cdot\mathbf{u}), 2\mathbf{n}_k\cdot\mathbf{v}\rangle_0\\
    &= \langle \Pi_{Q_h}(2 \mathbf{n}_k\cdot\mathbf{n}), (\Pi_{Q_h}+\kappa)(2\mathbf{n}_k\cdot\mathbf{v})\rangle_0 + \langle \kappa(2\mathbf{n}_k\cdot\mathbf{u}), (\Pi_{Q_h}+\kappa)(2\mathbf{n}_k\cdot\mathbf{v})\rangle_0\\
    &= \langle \Pi_{Q_h}(2 \mathbf{n}_k\cdot\mathbf{u}), \Pi_{Q_h}(2\mathbf{n}_k\cdot\mathbf{v})\rangle_0 + \langle\kappa(2\mathbf{n}_k\cdot\mathbf{u}),\kappa(2\mathbf{n}_k\cdot\mathbf{v})\rangle_0.
    \end{aligned}
\end{equation*}
Note here that the assembly of the first term is $B^\top M_\lambda^{-1}B$, where $M_\lambda$ is the mass matrix associated with the finite element space for the multiplier $Q_h$. This can then be readily used with the Sherman--Morrison--Woodbury formula to derive an approximation of the Schur complement.
The second term $\langle\kappa(2\mathbf{n}_k\cdot\mathbf{u}),\kappa(2\mathbf{n}_k\cdot\mathbf{v})\rangle_0$ characterizes the difference between ${A}_*$ and $B^\top M_\lambda^{-1}B$.
The next result shows that it vanishes as the mesh size $h\rightarrow 0$ (see Theorem \ref{thm:remaining-vanish}) and thus, in this limit, the tractable term $B^\top M_\lambda^{-1}B$ dominates ${A}_*$.

\begin{theorem}\label{thm:remaining-vanish}
    Let $(\delta\mathbf{n}_h, \delta\lambda_h)\in V_h\times Q_h$ be the solution of the augmented discrete system \eqref{eq:modified-newton-aug} with corresponding degrees of freedom $(U, P)\in\mathbb{R}^n\times\mathbb{R}^m$.
    Then, for the Newton linearization at a given approximation $(\mathbf{n}_k, \lambda_k)$ satisfying $\alpha \le |\mathbf{n}_k|^2\le \beta$ with $0< \alpha \le 1\le \beta$ and $|\nabla\mathbf{n}_k|$ bounded pointwise a.e., we have
    \begin{equation*}
        \left\|\left(A_* - B^\top M_\lambda^{-1}B\right)U\right\|_{\mathbb{R}^n} \le C h^{1+\frac{d}{2}}\|\delta\mathbf{n}_h\|_1,
    \end{equation*}
    where $\|\cdot\|_{\mathbb{R}^n}$ denotes the Euclidean norm.
\end{theorem}
\begin{proof}
    Assuming $\mathbf{v}_h\in V_h$ and using the basis representations in $V_h=\mathrm{span}\{\phi_i\}$ for $\delta\mathbf{n}_h$ and $\mathbf{v}_h$:
    \begin{equation*}
        \delta\mathbf{n}_h = \sum_{i=1}^{n} U_i\phi_i, \quad v_h = \sum_{i=1}^n Y_i\phi_i,
    \end{equation*}
    we obtain
    \begin{equation*}
        \begin{aligned}
            \left\|\left(A_* - B^\top M_\lambda^{-1}B\right)U\right\|_{\mathbb{R}^n}
            &= \sup_{\|Y\|_{\mathbb{R}^n}=1} Y^\top \left(A_* - B^\top M_\lambda^{-1}B\right)U\\
            &=\sup_{\|Y\|_{\mathbb{R}^n}=1} \langle\kappa(2\mathbf{n}_k\cdot\delta\mathbf{n}_h),\kappa(2\mathbf{n}_k\cdot\mathbf{v}_h)\rangle_0\\
            &\le \sup_{\|Y\|_{\mathbb{R}^n}=1} \|\kappa(2\mathbf{n}_k\cdot\delta\mathbf{n}_h)\|_0 \|\kappa(2\mathbf{n}_k\cdot\mathbf{v}_h)\rangle_0\|_0\\
            &\le \underbrace{\|\kappa\|}_{G_1} \underbrace{\sup_{\|Y\|_{\mathbb{R}^n}=1} \|2\mathbf{n}_k\cdot\mathbf{v}_h\|_0}_{G_2}  \underbrace{\|\kappa(2\mathbf{n}_k\cdot\delta\mathbf{n}_h)\|_0}_{G_3}
        \end{aligned}
    \end{equation*}
    by applying the Cauchy--Schwarz inequality.

    One readily sees that $G_1\le C_1$ for a certain constant $C_1$ from the continuity of $\kappa$.
    Furthermore, we write
    \begin{equation*}
        G_2 = \sup_{\mathbf{v}_h}\frac{\|2\mathbf{n}_k\cdot\mathbf{v}_h\|_0}{\|Y\|_{\mathbb{R}^n}}.
    \end{equation*}
    Note that \cite[Theorem 3.43]{knabner-book} as used in \cite{heister-2012-article} gives the relation between the discrete vector $Y$ and its associated continuous function $\mathbf{v}_h$:
    \begin{equation*}
        \|Y\|_{\mathbb{R}^n} \ge C_r h^{-\frac{d}{2}}\|\mathbf{v}_h\|_0,
    \end{equation*}
    for some $C_r>0$.
    Then with the fact that $\mathbf{n}_k$ is bounded we have
    \begin{equation*}
        G_2 \le \sup_{\mathbf{v}_h} \frac{\|2\mathbf{n}_k\cdot\mathbf{v}_h\|_0}{C_r h^{-\frac{d}{2}}\|\mathbf{v}_h\|_0} \le C_2 h^{\frac{d}{2}}.
    \end{equation*}
    Moreover, \cite[Theorem 1]{clement-1975-article} implies
    \begin{equation*}
        \|\kappa(p)\|_0 = \|p - \Pi_{Q_h}p\|_0 \le C_4 h\|p\|_1 \quad \text{for } p\in H^1(\Omega),
    \end{equation*}
    we deduce the following $L^2$-projection error estimate
    \begin{equation*}
        G_3 = \|\kappa(2\mathbf{n}_k\cdot\delta\mathbf{n}_h)\|_0 \le C_4 h \|2\mathbf{n}_k\cdot\delta\mathbf{n}_h\|_1 \le C_3 h \|\delta\mathbf{n}_h\|_1.
    \end{equation*}
    Note here we have used the pointwise boundedness of $\mathbf{n}_k, \nabla\mathbf{n}_k$ a.e.\ and the fact that $\delta\mathbf{n}_h\in V_h\subset H^1(\Omega)$.

    Combining these estimates regarding $G_1,G_2, G_3$, we find
    \begin{equation*}
        \left\|\left(A_* - B^\top M_\lambda^{-1}B\right)U\right\|_{\mathbb{R}^n} \le C h^{1+\frac{d}{2}}\|\delta\mathbf{n}_h\|_1 \to 0\quad \text{as } h\to 0.
    \end{equation*}
    \qed
\end{proof}

This result suggests the use of the algebraic approximation
\begin{equation}
    \label{eq:mod-schur}
    S_\gamma \approx -B\left(A+\gamma B^\top M_\lambda^{-1} B \right)^{-1}B^\top.
\end{equation}
The reason for doing so is that we can straightforwardly calculate the inverse of this
approximation \eqref{eq:mod-schur} by the Sherman--Morrison--Woodbury formula as follows:
\begin{equation*}
    S_\gamma^{-1} = -BA^{-1}B^\top - \gamma M_\lambda^{-1} = S^{-1} - \gamma M_\lambda^{-1}.
\end{equation*}
The solver requires the action of $S_\gamma^{-1}$, i.e., solving linear systems involving $S_\gamma$.
For large $\gamma$, a simple and effective approach is to employ the approximation
\begin{equation} \label{eq:mod-schur-inverse}
S_\gamma^{-1} \approx - \gamma M_\lambda^{-1}.
\end{equation}
On the infinite-dimensional level, the effect of the augmented Lagrangian term is to
make $-\gamma^{-1} I$ ($I$ the identity operator on the multiplier space) an effective approximation
for the Schur complement \cite[Lemma 3]{polyak1974}. When discretized, this indicates
that the weighted multiplier mass matrix $-\gamma^{-1} M_\lambda$ will be an effective approximation for
$S_\gamma$, with the approximation improving as $\gamma \to \infty$.

In fact, the approximation of the inverse of the discretely augmented Schur complement \eqref{eq:mod-schur-inverse} can be improved further
by combining $-\gamma M_{\lambda}^{-1}$ with a good approximation of the unaugmented Schur complement $S$ \cite{he2018}.
Given an approximation $\tilde{S}$ of $S$, we employ
\begin{equation} \label{eq:firstapprox}
S_\gamma^{-1} \approx \tilde{S}_\gamma^{-1} = \tilde{S}^{-1} - \gamma M_\lambda^{-1}.
\end{equation}
It is therefore of interest to consider the Schur complement of the unaugmented problem.
In the context of the Stokes equations, the Schur complement is spectrally equivalent to the viscosity-weighted pressure mass matrix \cite{silvester-1994-article, wathen-1991-article,elman-2005-book}.
Following similar techniques, an approximation can be obtained by proving that $BA^{-1}B^\top$ is spectrally equivalent to $M_{\lambda}$
for the equal-constant nematic case. This gives us good insight into the choice of $\tilde{S}^{-1}$.

\begin{theorem}\label{thm:schur-approx}
    For equal-constant nematic LC problems without augmented Lagrangian stabilization,
    the matrix $BA^{-1}B^\top$ arising from the Newton-linearized system is spectrally equivalent to the multiplier mass matrix $M_\lambda$, under the same assumptions as in Lemma \ref{lem:acoer:dis}.
\end{theorem}
\begin{proof}
    For the equal-constant model with Dirichlet boundary conditions $\mathbf{n}=\mathbf{g}\in H^{1/2}(\partial\Omega;\mathbb{S}^2)$,
    its corresponding Lagrangian is
    \begin{displaymath}
        {\mathcal{L}}(\mathbf{n},\lambda) = \frac{K}{2}\langle \nabla\mathbf{n},\nabla\mathbf{n}\rangle_0
        +\langle \lambda,\mathbf{n}\cdot\mathbf{n}-1\rangle_0.
    \end{displaymath}
    After Newton linearization and introducing conforming finite dimensional spaces $V_h\subset \mathbf{H}^1_0(\Omega)$ and $Q_h \subset L^2(\Omega)$,
    the discrete variational problem is to find $\delta\mathbf{n}_h\in V_h$, $\delta\lambda_h\in Q_h$ satisfying
    \begin{equation*}
        \begin{split}
            K \langle \nabla\delta\mathbf{n}_h,\nabla\mathbf{v}_h\rangle_0
            + &2\langle \lambda_k, \delta\mathbf{n}_h\cdot\mathbf{v}_h \rangle_0
            + 2\langle \delta\lambda_h, \mathbf{n}_k\cdot\mathbf{v}_h\rangle_0 \\
            &= -K\langle \nabla\mathbf{n}_k\cdot\nabla\mathbf{v}_h\rangle_0 -2\langle \lambda_k,\mathbf{n}_k\cdot\mathbf{v}_h\rangle_0
            \quad \forall \mathbf{v}_h\in V_h,\\
            2 \langle \mu_h, \mathbf{n}_k\cdot\delta\mathbf{n}_h \rangle_0 &= -\langle \mu_h,\mathbf{n}_k\cdot\mathbf{n}_k-1\rangle_0
        \quad\forall \mu_h\in Q_h,
        \end{split}
    \end{equation*}
    where $\mathbf{n}_k$ and $\lambda_k$ represent the current approximations to $\mathbf{n}$ and $\lambda$, respectively.
    This can be rewritten in block matrix form as
    \begin{equation*}
        \mathcal{A}
        \begin{bmatrix}
            U\\
            P
        \end{bmatrix} \coloneqq
        \begin{bmatrix}
            A & B^\top\\
            B & 0
        \end{bmatrix}
        \begin{bmatrix}
            U\\
            P
        \end{bmatrix}=
        \begin{bmatrix}
            f \\
            g
        \end{bmatrix},
    \end{equation*}
    where as before $U\in \mathbb{R}^n$ and $P\in \mathbb{R}^m$ are the unknown coefficients of the discrete director update and the discrete Lagrange multiplier update
    with respect to the basis functions in $V_h$ and $Q_h$,
    and $A$ denotes the symmetric form $K\langle \nabla\delta\mathbf{n}_h,\nabla \mathbf{v}_h\rangle_0 + 2\langle \lambda_k,\delta\mathbf{n}_h\cdot\mathbf{v}_h\rangle_0$.
    The coercivity property of the bilinear form from Lemma \ref{lem:acoer:dis} ensures that $A$ is positive definite.

    The coefficient matrix $\mathcal{A}$ is symmetric and indefinite
    (resulting in $\mathcal{A}$ possessing both positive and negative eigenvalues).
    Moreover, $\mathcal{A}$ is non-singular if and only if $B$ has full row rank, which can be deduced from the discrete inf-sup condition.

    Denote
    \begin{equation*}
        \begin{aligned}
            \|\mathbf{u}_h\|^2_{lc} &= K\langle \nabla\mathbf{u}_h,\nabla\mathbf{u}_h\rangle_0
        +\langle \lambda_k,2\mathbf{u}_h\cdot\mathbf{u}_h\rangle_0,\\
            \|{\mu}_h\|^2_0 &= \langle \mu_h,\mu_h\rangle_0.
        \end{aligned}
    \end{equation*}
    Notice that the validity of the first norm follows from the assumed pointwise non-negativity of $\lambda_k$.

    For a stable mixed finite element, from the inf-sup condition, there exists a positive constant $c$ independent of the mesh size $h$ such that
    \begin{equation*}
        \underset{\mathbf{u}_h\in V_h\backslash \{0\}}{\mathrm{sup}} \frac{\langle \mu_h, 2\mathbf{n}_k \cdot\mathbf{u}_h\rangle_0}{\|\mathbf{u}_h\|_{lc}} \ge c\|\mu_h\|_0
        \quad \forall \mu_h \in Q_h,
    \end{equation*}
    leading to its matrix form
    \begin{displaymath}
        \underset{U\in\mathbb{R}^n\backslash \{0\}}{\mathrm{max}} \frac{P^\top BU}{[U^\top AU]^{1/2}} \ge c[P^\top M_\lambda P]^{1/2}
        \quad \forall P\in\mathbb{R}^m.
    \end{displaymath}
    Thus, we have
    \begin{displaymath}
        \begin{aligned}
            c[P^\top M_\lambda P]^{1/2} &\le \underset{U\in\mathbb{R}^n\backslash \{0\}}{\mathrm{max}} \frac{P^\top BU}{[U^\top AU]^{1/2}}\\
            &= \underset{z=A^{1/2}U\ne 0}{\mathrm{max}}\frac{P^\top BA^{-1/2}z}{[z^\top z]^{1/2}} \\
            &= (P^\top B A^{-1}B^\top P)^{1/2}\quad \forall P\in \mathbb{R}^m,
        \end{aligned}
    \end{displaymath}
    where the maximum is attained at $z=(P^\top BA^{-1/2})^\top$.
    It yields
    \begin{equation}\label{sub1}
        c^2\frac{ P^\top M_\lambda P}{P^\top P} \le \frac{ P^\top B A^{-1}B^\top P}{P^\top P}
        \quad\forall P\in \mathbb{R}^m\backslash \{0\}.
    \end{equation}

    Regardless of the stability of the finite element pair, we can deduce from the boundedness of $B$ that there exists a positive constant $c_1$ such that
    \begin{displaymath}
        P^\top BU \le c_1 [P^\top M_\lambda P]^{1/2}[U^\top AU]^{1/2}\quad \forall U\in\mathbb{R}^n, \forall P\in \mathbb{R}^m.
    \end{displaymath}
    Hence,
    \begin{displaymath}
        \begin{aligned}
            c_1 [P^\top M_\lambda P]^{1/2} &\ge \underset{U\in\mathbb{R}^n\backslash \{0\}}{\mathrm{max}} \frac{P^\top BU}{[U^\top AU]^{1/2}}\\
            &= \underset{{z=A^{1/2}U\ne 0} }{\mathrm{max}}\frac{P^\top BA^{-1/2}z}{[z^\top z]^{1/2}} \\
            &= (P^\top B A^{-1}B^\top P)^{1/2} \quad \forall P\in \mathbb{R}^m,
        \end{aligned}
    \end{displaymath}
    where again the maximum is attained at $z=(P^\top BA^{-1/2})^\top$.
    This gives rise to
    \begin{equation}\label{sub2}
        \frac{P^\top B A^{-1}B^\top P}{P^\top M_\lambda P} \le c_1^2
        \quad \forall P\in\mathbb{R}^m\backslash \{0\}.
    \end{equation}

    Therefore for inf-sup stable finite element pairs, we have by \eqref{sub1} and \eqref{sub2}
    \begin{equation*}
        c^2\le \frac{P^\top B A^{-1}B^\top P}{P^\top M_\lambda P} \le c_1^2
        \quad \forall P\in\mathbb{R}^m\backslash \{0\}.
    \end{equation*}

    This indicates that $BA^{-1}B^\top$ is spectrally equivalent to $M_\lambda$. \qed
\end{proof}

\begin{remark}
    It follows from Theorem \ref{thm:schur-approx} that $\gamma=0$ should show mesh-independence (i.e., the average number of FGMRES iterations per Newton iteration does not deteriorate as one refines the mesh) in the case of equal-constant nematic LC.
    This can be observed in subsequent numerical experiments reported in Table \ref{table:allu-ell} (see the column where $\gamma=0$).
    One should also notice that such mesh-independence for $\gamma=0$ is also shown in Table \ref{table:allu-rec} for the non-equal-constant case,
    suggesting it has use outside the context of augmented Lagrangian methods also.
\end{remark}

Combining Theorem \ref{thm:schur-approx} with \eqref{eq:firstapprox}, our final approximation for $S_\gamma^{-1}$ is given by
\begin{equation}\label{eq:aug-schur-approx}
    S_\gamma^{-1} \approx \tilde{S}_\gamma^{-1} = -(1+\gamma) M_{\lambda}^{-1}.
\end{equation}

%% file: improv.tex
\section{Improvement of the constraint}\label{sec:improv}

We have now observed that the continuous AL form introduced in Section \ref{sec:cont-penal} can help control the Schur complement.
Another contribution of this AL stabilization is that it improves the discrete constraint as we increase the value of the penalty parameter $\gamma$.
An example of improving the linear divergence-free constraint in the Stokes system can be found in \cite[Section 5.1]{john-2017-article}.
In this section, we will use a similar strategy to show the improvement of the discrete constraint as $\gamma$ increases.

We restrict ourselves to the equal-constant case with \emph{constant} Dirichlet boundary conditions.
That is to say, we consider the Oseen--Frank model with Dirichlet boundary condition $\mathbf{n}|_{\partial\Omega}=\mathbf{g}$,
where $\mathbf{g}$ is a nonzero constant vector satisfying $|\mathbf{g}|=1$.
We use the $[\mathbb{P}_1]^d$-$\mathbb{P}_1$ finite element pair in this section,
so both the director $\mathbf{n}$ and the Lagrange multiplier $\lambda$ are approximated by continuous piecewise-linear polynomials.
For this section, we denote finite element spaces for the director and the Lagrange multiplier by $V_{h,g}\coloneqq V_h \cap\mathbf{H}^1_g(\Omega)$ and $Q_h \subset L^2(\Omega)$, respectively,
and denote $V_{h,0}= V_h\cap\mathbf{H}^1_0(\Omega)$.

We restate the associated nonlinear discrete variational problem as follows:
find $(\mathbf{n}_h, \lambda_h)\in V_{h,g}\times Q_h$ such that
\begin{subequations}\label{eq:aug-var}
    \begin{equation}\label{eq:1var}
        \begin{aligned}
            K &\langle \nabla\mathbf{n}_h, \nabla\mathbf{v}_h\rangle_0
            + Kq_0\langle \mathbf{v}_h,\nabla\times\mathbf{n}_h\rangle_0
            + Kq_0\langle \mathbf{n}_h, \nabla\times\mathbf{v}_h\rangle_0\\
            &+ 2\langle \lambda_h, \mathbf{n}_h\cdot\mathbf{v}_h\rangle_0
            + 2\gamma \langle \mathbf{n}_h\cdot\mathbf{n}_h-1, \mathbf{n}_h\cdot\mathbf{v}_h \rangle_0
            = 0 \quad \forall \mathbf{v}_h\in V_{h,0},
        \end{aligned}
    \end{equation}
    \begin{equation}\label{eq:2var}
        \langle \mu_h, \mathbf{n}_h\cdot\mathbf{n}_h-1\rangle_0
        = 0 \quad \forall \mu_h \in Q_h.
    \end{equation}
\end{subequations}
Take the test function $\mathbf{v}_h = \mathbf{n}_h-{\mathbf{g}}\in V_{h,0}$ in \eqref{eq:1var} to obtain
\begin{equation}\label{eq:ln}
    \begin{aligned}
        K \|\nabla\mathbf{n}_h\|^2_0
        &+ 2Kq_0\langle \mathbf{n}_h,\nabla\times\mathbf{n}_h\rangle_0
        + 2\langle \lambda_h, \mathbf{n}_h\cdot\mathbf{n}_h\rangle_0
        + 2\gamma \langle \mathbf{n}_h\cdot\mathbf{n}_h-1, \mathbf{n}_h\cdot\mathbf{n}_h \rangle_0\\
        &= Kq_0\langle \mathbf{g}, \nabla\times\mathbf{n}_h\rangle_0
        + 2\langle \lambda_h, \mathbf{n}_h\cdot\mathbf{g}\rangle_0
        + 2\gamma \langle \mathbf{n}_h\cdot\mathbf{n}_h-1, \mathbf{n}_h\cdot\mathbf{g} \rangle_0.
    \end{aligned}
\end{equation}
Note that in this step we have used the fact that since $\mathbf{g}$ is a constant vector,
its derivative is zero.

As \eqref{eq:2var} is valid for arbitrary $\mu_h\in Q_h$ and one can easily verify that $\mathbf{n}_h\cdot\mathbf{g}\in Q_h$,
we have
\begin{displaymath}
    \langle \mathbf{n}_h\cdot\mathbf{g}, \mathbf{n}_h\cdot\mathbf{n}_h-1 \rangle_0 =0.
\end{displaymath}
Then taking $\mu_h=1$ and $\mu_h=\lambda_h$ leads to
\begin{displaymath}
    \langle 1, \mathbf{n}_h\cdot\mathbf{n}_h-1\rangle_0 = 0\ \text{and}\
    \langle \lambda_h, \mathbf{n}_h\cdot\mathbf{n}_h-1\rangle_0 = 0,
\end{displaymath}
respectively.
Thus, \eqref{eq:ln} collapses to
\begin{equation}\label{eq:1}
    \begin{aligned}
        K \|\nabla\mathbf{n}_h\|^2_0
        + &2Kq_0\langle \mathbf{n}_h,\nabla\times\mathbf{n}_h\rangle_0
        + 2\langle \lambda_h, 1\rangle_0
        + 2\gamma \| \mathbf{n}_h\cdot\mathbf{n}_h-1 \|^2_0\\
        &= Kq_0\langle \mathbf{g}, \nabla\times\mathbf{n}_h\rangle_0
        + 2\langle \lambda_h, \mathbf{n}_h\cdot\mathbf{g}\rangle_0.
    \end{aligned}
\end{equation}
By the Cauchy--Schwarz and H\"{o}lder inequalities, we observe an upper bound for the right hand side of \eqref{eq:1}:
\begin{equation}\label{eq:2}
    \begin{aligned}
        Kq_0\langle \mathbf{g}, \nabla\times\mathbf{n}_h\rangle_0
        + 2\langle \lambda_h, \mathbf{n}_h\cdot\mathbf{g}\rangle_0
        &\le Kq_0 \|\nabla\times\mathbf{n}_h\|_0
        + 2 \|\lambda_h\|_0\|\mathbf{n}_h\|_0\\
        &\le \frac{Kq_0}{2}+\frac{Kq_0}{2}\|\nabla\times\mathbf{n}_h\|_0^2
        +\|\lambda_h\|^2_0 +\|\mathbf{n}_h\|^2_0.
    \end{aligned}
\end{equation}
Meanwhile, the left hand side of \eqref{eq:1} can be bounded from below:
\begin{equation}\label{eq:lhs}
    \begin{aligned}
        K \|\nabla\mathbf{n}_h\|^2_0
        &+ 2Kq_0\langle \mathbf{n}_h,\nabla\times\mathbf{n}_h\rangle_0
        + 2\langle \lambda_h, 1\rangle_0
        + 2\gamma \| \mathbf{n}_h\cdot\mathbf{n}_h-1 \|^2_0\\
        &\ge K\|\nabla\mathbf{n}_h\|_0^2
        -2Kq_0|\langle\mathbf{n}_h,\nabla\times\mathbf{n}_h\rangle_0|
        -2|\langle \lambda_h,1\rangle_0|
        + 2\gamma \| \mathbf{n}_h\cdot\mathbf{n}_h-1 \|^2_0\\
        &\ge K\|\nabla\mathbf{n}_h\|_0^2 - Kq_0\|\mathbf{n}_h\|_0^2 - Kq_0\|\nabla\times\mathbf{n}_h\|^2_0
        -\|\lambda_h\|^2_0 -|\Omega|
        + 2\gamma \| \mathbf{n}_h\cdot\mathbf{n}_h-1 \|^2_0,
    \end{aligned}
\end{equation}
where $|\Omega|$ denotes the measure of the domain $\Omega$.

Hence, by combining \eqref{eq:2} and \eqref{eq:lhs}, we have
\begin{equation}\label{eq:est}
    \begin{aligned}
        K\|\nabla\mathbf{n}_h\|^2_0
        - &(Kq_0+1)\|\mathbf{n}_h\|_0^2
        - \frac{3}{2}Kq_0\|\nabla\times\mathbf{n}_h\|^2_0 \\
        &-\|\lambda_h\|^2_0 + 2\gamma\|\mathbf{n}_h\cdot\mathbf{n}_h-1\|^2_0
        \le \frac{Kq_0}{2} + |\Omega|.
    \end{aligned}
\end{equation}
Since the right hand side of \eqref{eq:est} is a fixed constant independent of $\gamma$,
taking $\gamma$ larger value forces the constraint approximation error $\|\mathbf{n}_h\cdot\mathbf{n}_h-1\|_0$ to become smaller.
In fact, \eqref{eq:est} implies that $\|\mathbf{n}_h\cdot\mathbf{n}_h-1\|_0\le \mathcal{O}(\gamma^{-1/2})$.

\begin{remark}
    The technique shown in this section can be extended in a similar way to the multi-constant case; we omit the details here for brevity.
\end{remark}

%% file: mg.tex
\section{A robust multigrid method for $A_{\gamma}$}\label{sec:mg}

As discussed in Section \ref{sec:schur}, the addition of the augmented Lagrangian term has the effect of controlling the Schur complement of the matrix in \eqref{eq:saddle-aug}.
However, the tradeoff is that it complicates the solution of the top-left block $A_\gamma$, as it adds a semi-definite term with a large coefficient.
For the augmented Lagrangian strategy to be successful, we require a $\gamma$-robust solver for the
top-left block. Fortunately, a rich literature is available to guide the development of multigrid
solvers for nearly singular systems \cite{schoberl-1999-article, schoberl-1999-phd-thesis, lee-2007-article}.
In this section we develop a parameter-robust multigrid method for $A_\gamma$.

Sch\"{o}berl's seminal paper on the construction of parameter-robust multigrid schemes \cite{schoberl-1999-article} lists
two requirements that must be satisfied for robustness.
The first requirement is a parameter-robust relaxation method; this is achieved by developing a space decomposition
that stably captures the kernel of the semi-definite terms.
The second requirement is a parameter-robust prolongation operator, i.e.~one whose continuity constant is independent of
the parameters. This is achieved by (approximately) mapping kernel functions on coarse grids to kernel functions on fine grids.
We discuss both of these requirements below.

For ease of notation, we consider the two-grid method applied to the equal-constant nematic case, and use subscripts $h$ and $H$ to distinguish fine and coarse levels respectively.
That is to say, $V_H$ represents the coarse-grid function space and $A_{H,\gamma}:V_H\rightarrow V_H^*$ corresponds to the partial differential equations (PDEs) on $V_H$.

For the domain $\Omega$, we consider a non-overlapping triangulation $\mathcal{M}_H$, i.e.,
\begin{equation*}
    \cup_{T\in \mathcal{M}_H} T = \bar{\Omega}\ \text{and}\
    \mathrm{int}(T_i)\cap \mathrm{int}(T_j)=\emptyset \quad\forall T_i\ne T_j, \ T_i, T_j \in \mathcal{M}_H.
\end{equation*}
The fine grid $\mathcal{M}_h$ with $h=H/2$ is obtained by a regular refinement of the simplices in $\mathcal{M}_H$.
In what follows we consider both the $[\mathbb{P}_1]^d$-$\mathbb{P}_1$ and $[\mathbb{P}_2]^d$-$\mathbb{P}_1$ discretizations.

\subsection{Relaxation}\label{sec:relax}

After applying the AL method introduced in Section \ref{sec:cont-penal}, the discrete linear variational form corresponding to the top-left block
$A_{\gamma} = A+\gamma A_*$ is given by
\begin{equation}\label{eq:ah}
    \begin{aligned}
        a^m(\mathbf{u}_h,\mathbf{v}_h) &\coloneqq K \langle \nabla\mathbf{u}_h,\nabla\mathbf{v}_h\rangle_0
        +2 \langle \lambda_k, \mathbf{u}_h\cdot\mathbf{v}_h\rangle_0\\
        &\quad + 4\gamma\langle \mathbf{n}_k\cdot\mathbf{u}_h, \mathbf{n}_k\cdot\mathbf{v}_h \rangle_0,
    \end{aligned}
\end{equation}
with $\mathbf{u}_h\in V_h\subset\mathbf{H}^1_0(\Omega)$ being the trial function and $\mathbf{v}_h\in V_h$ the test function.
Note that $\mathbf{n}_k$ and $\lambda_k$ are the current approximations to the director $\mathbf{n}$
and the Lagrange multiplier $\lambda$, respectively, in the Newton iteration.
The first two terms of $a^m$ are symmetric and coercive
because of the uniform non-negativity of $\lambda_k$ in the assumption of our well-posedness result.
The kernel of the semi-definite term involving $\gamma$ is
\begin{equation}\label{eq:N*}
    \mathcal{N}_h =\{\mathbf{u}_h\in V_h: \mathbf{n}_k\cdot\mathbf{u}_h=0\ \text{a.e.}\}.
\end{equation}
In the case of $\gamma$ being very large, the variational problem involving \eqref{eq:ah} is nearly singular and common relaxation methods like Jacobi and Gauss--Seidel will not yield effective multigrid cycles, as we explain below.

Relaxation schemes can be devised in a generic way by considering \emph{space decompositions}
\begin{equation}\label{eq:decomp}
    {V}_h  = \sum_{i=1}^M {V}_i,
\end{equation}
where the sum of vector spaces on the right is not necessarily a direct sum \cite{xu-1992-article}.
This space decomposition induces a relaxation method by (approximately) solving the Galerkin projection
of the error equation onto each subspace $V_i$, and combining the resulting estimates of the error.
This can be done in an additive or multiplicative way. For example, if $V_h = \mathrm{span}(\phi_1, \dots, \phi_N)$, Jacobi and Gauss--Seidel are induced by the space decomposition
\begin{equation} \label{eq:space-decomp-jacobi}
V_i = \mathrm{span}(\phi_i),
\end{equation}
where the updates are performed additively for Jacobi and multiplicatively for Gauss--Seidel.
One of the key insights of \cite{schoberl-1999-article, lee-2007-article} was that the key requirement for parameter-robustness when applied to nearly singular problems is that the space decomposition must satisfy the
\emph{kernel-capturing property}
\begin{equation}\label{eq:kernel-capture}
    \mathcal{N}_h= \sum_{i=1}^M (V_i\cap \mathcal{N}_h),
\end{equation}
that is, any kernel function can be written as a sum of kernel functions drawn from the subspaces.
In particular, each subspace $V_i$ must be rich enough to support kernel functions; in our context, this is not satisfied by the choice \eqref{eq:space-decomp-jacobi}, accounting for its poor behaviour as $\gamma \to \infty$.

In the mesh triangulation $\mathcal{M}_h$,
we denote the \textit{star} of a vertex $v_i$ as the patch of elements sharing $v_i$, i.e.,
\begin{displaymath}
    \mathrm{star}(v_i)\coloneqq \underset{T\in \mathcal{M}_h:v_i\in T}{\bigcup}T.
\end{displaymath}
This induces an associated space decomposition, called the \emph{star patch}, by
\begin{equation*}
    V_i\coloneqq \{\mathbf{u}_h\in V_h: \mathrm{supp}(\mathbf{u}_h)\subset \mathrm{star}(v_i)\}.
\end{equation*}
This is illustrated in Figure \ref{fig:patch} (left).
We call the induced relaxation method a \textit{star iteration}. In effect, each subspace solve solves for the
degrees of freedom in the interior of the patch of cells, with homogeneous Dirichlet conditions on the boundary of the patch.
Given a vertex or edge midpoint $v_i$, we denote the \textit{point-block} patch $V_i$ as the span of the basis functions associated with degrees of freedom that evaluate a function at $v_i$ (see Figure \ref{fig:patch}, middle). The induced relaxation method solves for all colocated degrees of freedom simultaneously.
These two space decompositions coincide for the $[\mathbb{P}_1]^d$-$\mathbb{P}_1$ discretization.

\begin{figure}[!ht]
    \centering
    \begin{minipage}{0.3\textwidth}
        \centering
        \begin{tikzpicture}[scale=1.5]
            \draw [black, fill] (1,1) circle (1.5pt);
            \draw [black, fill] (1,0) circle (1.5pt);
            \draw [black, fill] (0,1) circle (1.5pt);
            \draw [black, fill] (0,0) circle (1.5pt);
            \draw [black, fill] (1,2) circle (1.5pt);
            \draw [black, fill] (2,1) circle (1.5pt);
            \draw [black, fill] (2,2) circle (1.5pt);

            \draw [black, fill] (0.5,0) circle (1.5pt);
            \draw [black, fill] (0,0.5) circle (1.5pt);
            \draw [black, fill] (0.5,0.5) circle (1.5pt);
            \draw [black, fill] (1,0.5) circle (1.5pt);
            \draw [black, fill] (1.5,0.5) circle (1.5pt);
            \draw [black, fill] (0.5,1) circle (1.5pt);
            \draw [black, fill] (1.5,1) circle (1.5pt);
            \draw [black, fill] (0.5,1.5) circle (1.5pt);
            \draw [black, fill] (1,1.5) circle (1.5pt);
            \draw [black, fill] (1.5,1.5) circle (1.5pt);
            \draw [black, fill] (2,1.5) circle (1.5pt);
            \draw [black, fill] (1.5,2) circle (1.5pt);

            \draw [black, line width=0.5mm] (0,0) -- (1, 0);
            \draw [black, line width=0.5mm] (0,0) -- (0,1);
            \draw [black, line width=0.5mm] (0,0) -- (2,2);
            \draw [black, line width=0.5mm] (0,1) -- (2, 1);
            \draw [black, line width=0.5mm] (1,0) -- (1,2);
            \draw [black, line width=0.5mm] (0,1) -- (1,2);
            \draw [black, line width=0.5mm] (1,2) -- (2, 2);
            \draw [black, line width=0.5mm] (2,2) -- (2,1);
            \draw [black, line width=0.5mm] (1,0) -- (2,1);

            \draw[rotate around={45:(1,1)},blue, line width=0.5mm] (1,1) ellipse (30pt and 15pt);
        \end{tikzpicture}
    \end{minipage}\hfill
    \begin{minipage}{0.3\textwidth}
        \centering
        \begin{tikzpicture}[scale=1.5]
            \draw [black, fill] (1,1) circle (1.5pt);
            \draw [black, fill] (1,0) circle (1.5pt);
            \draw [black, fill] (0,1) circle (1.5pt);
            \draw [black, fill] (0,0) circle (1.5pt);
            \draw [black, fill] (1,2) circle (1.5pt);
            \draw [black, fill] (2,1) circle (1.5pt);
            \draw [black, fill] (2,2) circle (1.5pt);

            \draw [black, fill] (0.5,0) circle (1.5pt);
            \draw [black, fill] (0,0.5) circle (1.5pt);
            \draw [black, fill] (0.5,0.5) circle (1.5pt);
            \draw [black, fill] (1,0.5) circle (1.5pt);
            \draw [black, fill] (1.5,0.5) circle (1.5pt);
            \draw [black, fill] (0.5,1) circle (1.5pt);
            \draw [black, fill] (1.5,1) circle (1.5pt);
            \draw [black, fill] (0.5,1.5) circle (1.5pt);
            \draw [black, fill] (1,1.5) circle (1.5pt);
            \draw [black, fill] (1.5,1.5) circle (1.5pt);
            \draw [black, fill] (2,1.5) circle (1.5pt);
            \draw [black, fill] (1.5,2) circle (1.5pt);

            \draw [black, line width=0.5mm] (0,0) -- (1, 0);
            \draw [black, line width=0.5mm] (0,0) -- (0,1);
            \draw [black, line width=0.5mm] (0,0) -- (2,2);
            \draw [black, line width=0.5mm] (0,1) -- (2, 1);
            \draw [black, line width=0.5mm] (1,0) -- (1,2);
            \draw [black, line width=0.5mm] (0,1) -- (1,2);
            \draw [black, line width=0.5mm] (1,2) -- (2, 2);
            \draw [black, line width=0.5mm] (2,2) -- (2,1);
            \draw [black, line width=0.5mm] (1,0) -- (2,1);

            \draw[blue, line width=0.5mm] (0.5,0.5) circle (5pt);
            \draw[blue, line width=0.5mm] (1,0.5) circle (5pt);
            \draw[blue, line width=0.5mm] (0.5,1) circle (5pt);
            \draw[blue, line width=0.5mm] (1,1) circle (5pt);
            \draw[blue, line width=0.5mm] (1.5,1) circle (5pt);
            \draw[blue, line width=0.5mm] (1,1.5) circle (5pt);
            \draw[blue, line width=0.5mm] (1.5,1.5) circle (5pt);
            \draw[blue, line width=0.5mm] (1,0) circle (5pt);
            \draw[blue, line width=0.5mm] (0,1) circle (5pt);
            \draw[blue, line width=0.5mm] (0,0.5) circle (5pt);
            \draw[blue, line width=0.5mm] (1,2) circle (5pt);
            \draw[blue, line width=0.5mm] (2,1) circle (5pt);
            \draw[blue, line width=0.5mm] (2,2) circle (5pt);
            \draw[blue, line width=0.5mm] (0.5,0) circle (5pt);
            \draw[blue, line width=0.5mm] (0.5,1.5) circle (5pt);
            \draw[blue, line width=0.5mm] (0,0) circle (5pt);
            \draw[blue, line width=0.5mm] (1.5,0.5) circle (5pt);
            \draw[blue, line width=0.5mm] (2,1.5) circle (5pt);
            \draw[blue, line width=0.5mm] (1.5,2) circle (5pt);

        \end{tikzpicture}
    \end{minipage}\hfill
    \begin{minipage}{0.3\textwidth}
        \centering
        \begin{tikzpicture}[scale=1.5]
            \draw [black, fill] (1,1) circle (1.5pt);
            \draw [black, fill] (1,0) circle (1.5pt);
            \draw [black, fill] (0,1) circle (1.5pt);
            \draw [black, fill] (0,0) circle (1.5pt);
            \draw [black, fill] (1,2) circle (1.5pt);
            \draw [black, fill] (2,1) circle (1.5pt);
            \draw [black, fill] (2,2) circle (1.5pt);

            \draw [black, line width=0.5mm] (0,0) -- (1, 0);
            \draw [black, line width=0.5mm] (0,0) -- (0,1);
            \draw [black, line width=0.5mm] (0,0) -- (2,2);
            \draw [black, line width=0.5mm] (0,1) -- (2, 1);
            \draw [black, line width=0.5mm] (1,0) -- (1,2);
            \draw [black, line width=0.5mm] (0,1) -- (1,2);
            \draw [black, line width=0.5mm] (1,2) -- (2, 2);
            \draw [black, line width=0.5mm] (2,2) -- (2,1);
            \draw [black, line width=0.5mm] (1,0) -- (2,1);

            \draw[blue, line width=0.5mm] (1,1) circle (5pt);
            \draw[blue, line width=0.5mm] (1,0) circle (5pt);
            \draw[blue, line width=0.5mm] (0,1) circle (5pt);
            \draw[blue, line width=0.5mm] (0,0) circle (5pt);
            \draw[blue, line width=0.5mm] (1,2) circle (5pt);
            \draw[blue, line width=0.5mm] (2,1) circle (5pt);
            \draw[blue, line width=0.5mm] (2,2) circle (5pt);

        \end{tikzpicture}
    \end{minipage}\hfill
    \caption{Illustrations of the star patch of the center vertex (left) and the point-block patch (middle) for the finite element pair $[\mathbb{P}_2]^2$-$\mathbb{P}_1$.
        Note that these two patches (right) are the same for $[\mathbb{P}_1]^2$-$\mathbb{P}_1$ discretization.
    Here, black dots represent the degrees of freedom, and the blue lines gather degrees of freedom solved for simultaneously in the relaxation.}
    \label{fig:patch}
\end{figure}
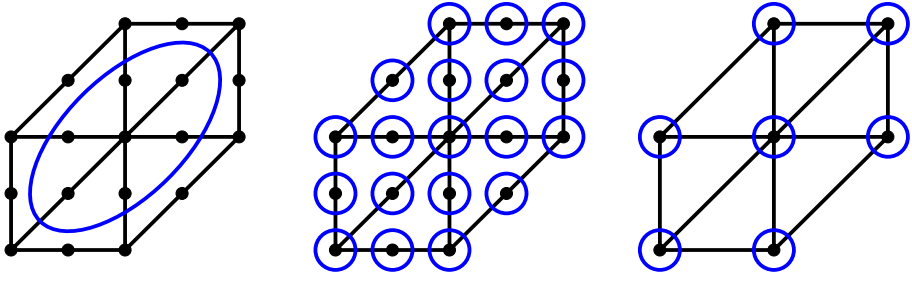

We now briefly explain why these two decompositions approximately satisfy the kernel-capturing condition \eqref{eq:kernel-capture} for the finite element pair $[\mathbb{P}_1]^d$-$\mathbb{P}_1$.
First, we define an approximate kernel
\begin{equation}\label{eq:Nhat}
    \tilde{\mathcal{N}}_h =\{\mathbf{u}_h\in V_h: \mathbf{n}_k\cdot \mathbf{u}_h=0\ \text{on each vertex}\}.
\end{equation}
Since $\mathbf{n}_k$ is the current approximation to the director $\mathbf{n}$, we have
$\mathbf{n}_k\in V_h=\sum_i V_i$.
We are therefore able to express $\mathbf{n}_k$ as $\mathbf{n}_k=\sum_i \mathbf{n}_k^i$,
where $\mathbf{n}_k^i\in V_i$ describes the function at the vertex $v_i$.
Similarly, we split $\mathbf{u}_h$ into $\mathbf{u}_h=\sum_i \mathbf{u}_h^i$ with $\mathbf{u}_h^i\in V_i$.
For each vertex $v_i$, the requirement $u_h \in \tilde{\mathcal{N}}_h$ yields
\begin{equation}\label{eq:local-cond}
    \mathbf{n}_k^i\cdot\mathbf{u}_h^i=0 \quad\forall i.
\end{equation}
The definition of $V_i$ ensures that $\mathbf{u}_h^i$ and $\mathbf{n}_k^i$ are only supported on the interior of the star of $v_i$.
We deduce that on each vertex
\begin{equation*}
    \mathbf{n}_k^j\cdot\mathbf{u}_h^i =0 \quad\forall i\ne j,
\end{equation*}
which yields
$ \sum_j\mathbf{n}_k^j\cdot\mathbf{u}_h^i =\mathbf{n}_k\cdot\mathbf{u}_h^i=0$.
Hence, $\mathbf{u}_h^i\in \tilde{\mathcal{N}}_h \forall i$ and we obtain the kernel-capturing condition \eqref{eq:kernel-capture} for the approximate kernel $\tilde{\mathcal{N}}_h$.

For the $[\mathbb{P}_2]^d$-$\mathbb{P}_1$ finite element pair, the satisfaction of the kernel-capturing property for the approximate kernel follows along similar lines.
For the point-block patch, \eqref{eq:local-cond} still holds.
The star patch uses larger subspaces, each one including multiple point-block patches, but it can be easily verified that \eqref{eq:local-cond} is still fulfilled.

\subsubsection{Robustness analysis of the approximate kernel}

While we are not able to prove the kernel capturing property for the kernel \eqref{eq:N*},
we can still obtain the spectral inequalities
\begin{equation}\label{eq:spectral}
    c_1 D_{h,\gamma} \le A_{h,\gamma} \le c_2 D_{h,\gamma},
\end{equation}
when using the approximate kernel \eqref{eq:Nhat}.
Here, $D_{h,\gamma}$ is the preconditioner to be specified later for the operator $A_{h,\gamma}$ and $C\le D$ represents $\|\mathbf{u}\|_C \le \|\mathbf{u}\|_D$ for all $\mathbf{u}$.
We prove that $c_1$ depends on $\gamma$, but the dependence can be well controlled so that the preconditioner is not badly affected by varying $\gamma$, while $c_2$ is always independent of $\gamma$.
For simplicity, we prove the case for the equal-constant nematic case with the $[\mathbb{P}_1]^d$-$\mathbb{P}_1$ discretization; extensions to the non-equal-constant cholesteric case 
and to the $[\mathbb{P}_2]^d$-$\mathbb{P}_1$ discretization are possible.

We define the operator associated to $a^m$, $A_{h,\gamma}: V_h\rightarrow V_h^*$, by
\begin{equation*}
    \langle A_{h,\gamma}\mathbf{u}_h, \mathbf{v}_h \rangle_0 \coloneqq a^m(\mathbf{u}_h, \mathbf{v}_h).
\end{equation*}

For the space decomposition $V_h= \sum_i V_i$, we denote the lifting operator (the natural inclusion) by $I_i:V_i\rightarrow V_h$
and choose the Galerkin subspace operator $A_i: V_i\rightarrow V_i$ to satisfy
\begin{equation*}
    \langle A_i \mathbf{u}_i, \mathbf{v}_i\rangle_0 \coloneqq \langle A_{h,\gamma}I_i\mathbf{u}_i, I_i\mathbf{v}_i\rangle_0
    \quad \forall \mathbf{u}_i, \mathbf{v}_i\in V_i.
\end{equation*}
This implies that $A_i = I_i^*A_{h,\gamma}I_i$.

The additive Schwarz preconditioner $D_{h,\gamma}$ for a problem $A_{h,\gamma}w_h=d_h$
associated with the space decomposition \eqref{eq:decomp}
is defined by the action of its inverse \cite{xu-1992-article}:
\begin{equation*}
    w_h = D_{h,\gamma}^{-1} d_h
\end{equation*}
given by
\begin{equation*}
    w_h = \sum_{i=1}^M I_i w_i,
\end{equation*}
with $w_i\in V_i$ being the unique solution of
\begin{equation*}
    \langle A_i w_i, v_i\rangle_0 = \langle d_h, I_iv_i\rangle_0 \quad \forall v_i\in V_i.
\end{equation*}
Hence, we can rewrite the preconditioning operator $D_{h,\gamma}^{-1}$ in operator form as
\begin{equation*}
    D_{h,\gamma}^{-1} = \sum_{i=1}^M I_iA_i^{-1}I^*_i.
\end{equation*}

We now state for completeness a classical result in the analysis of additive Schwarz preconditioners,
see e.g.~\cite[Theorem 3.1]{schoberl-1999-phd-thesis} and the references therein.
\begin{theorem}\label{thm:equalnorm}
    Define the splitting norm for $\mathbf{u}_h \in V_h$ as
    \begin{equation*}
    \vertiii{\mathbf{u}_h}^2 \coloneqq \underset{\substack{\mathbf{u}_h=\sum_i I_i\mathbf{u}_i \\ \mathbf{u}_i\in V_i}}{\inf} \sum_{i=1}^M \|\mathbf{u}_i\|_{A_i}^2.
    \end{equation*}
    This splitting norm is equal to the norm $\|\mathbf{u}_h\|_{D_{h,\gamma}}\coloneqq \langle D_{h,\gamma} \mathbf{u}_h, \mathbf{u}_h\rangle_0^{1/2}$
    generated by the additive Schwarz preconditioner,
    i.e.~it holds that
    \begin{equation*}
    \vertiii{\mathbf{u}_h}^2 = \|\mathbf{u}_h\|^2_{D_{h,\gamma}} \quad \forall \mathbf{u}_h\in V_h.
    \end{equation*}
\end{theorem}

To build intuition, let us examine why Jacobi relaxation defined by the space decomposition \eqref{eq:space-decomp-jacobi} is not robust as $\gamma \to \infty$.
With \eqref{eq:space-decomp-jacobi}, the decomposition $\mathbf{u}_h=\sum_i\mathbf{u}_i, \mathbf{u}_i\in V_i$ is unique.
It yields that
\begin{equation}\label{eq:j}
    \begin{aligned}
        \| \mathbf{u}_h\|^2_{D_{h,\gamma}}
        &= \vertiii{\mathbf{u}_h}^2
        = \sum_i \langle A_i\mathbf{u}_i, \mathbf{u}_i\rangle_0
        = \sum_i \langle A_{h,\gamma}\mathbf{u}_i, \mathbf{u}_i\rangle_0 \\
        &\preceq (1+\gamma) \sum_i\|\mathbf{u}_i\|_1^2
        \preceq \frac{1+\gamma}{h^2} \sum_i \|\mathbf{u}_i\|_0^2
        \preceq \frac{1+\gamma}{h^2} \|\mathbf{u}_h\|_0^2 \\
        &\preceq \frac{1+\gamma}{h^2} \|\mathbf{u}_h\|^2_{A_{h,\gamma}},
    \end{aligned}
\end{equation}
where $a\preceq b$ means that there exists a constant $c$ independent of $a$ and $b$ such that $a\le cb$.
Note that the bound in \eqref{eq:j} is parameter-dependent and deteriorates as $\gamma\rightarrow \infty$ or $h\rightarrow 0$.

In order to deduce the robustness result for our approximate kernel \eqref{eq:Nhat}, we first derive the following lemma.
\begin{lemma}\label{lem1}
    Let $\mathbf{u}_0=\sum_i \mathbf{u}_0^i \in \tilde{\mathcal{N}}_h$ and assume $\mathbf{n}_k\in [\mathbb{P}_1]^d$.
    Then it holds that
    \begin{equation*}
        \sum_i \| \mathbf{u}_0^i \cdot \mathbf{n}_k\|^2_{L^2(\Omega)} \preceq h^2 \|\mathrm{D}\mathbf{n}_k\|^2_{L^\infty(\Omega)} \| \mathbf{u}_0\|^2_{L^2(\Omega)},
    \end{equation*}
    where $\mathrm{D}\mathbf{n}_k$ denotes the Jacobian matrix of $\mathbf{n}_k$.
\end{lemma}
\begin{proof}
Consider the vertex $v_i$ on the boundary of an element $T$.
As $ \mathbf{n}_k\in [\mathbb{P}_1]^d$, we have
\begin{equation*}
    (\mathbf{u}_0^i \cdot \mathbf{n}_k) (x) = \mathbf{u}_0^i(x) \cdot \mathbf{n}_k(v_i) + \mathbf{u}^i_0(x) \cdot [ \mathrm{D} \mathbf{n}_k(v_i)(x-v_i)] \quad \forall x\in T.
\end{equation*}
Note that $\mathbf{u}_0^i \cdot \mathbf{n}_k$ vanishes at the vertex $v_i$ as $\mathbf{u}_0 \in \tilde{\mathcal{N}}_h$.
Moreover, we know that $\mathbf{u}_0^i(x)/\|\mathbf{u}_0^i(x)\|$ is constant on the interior of the patch around $v_i$, and $\mathbf{u}_0^i(x)$ is zero on the boundary of the patch,
since we can write $\mathbf{u}_0^i(x)=\mathbf{u}_0(v_i) \psi_i(x)$ with $\psi_i$ denoting the scalar piecewise linear basis function (vanishing outside the patch) associated with $v_i$.
Therefore, we can deduce $\mathbf{u}_0^i(x) \cdot \mathbf{n}_k(v_i)=0$ on $T$.
In addition, we have $\|x-v_i\| \preceq h$ on the element $T$. We thus conclude that
\begin{equation*}
    \| \mathbf{u}_0^i \cdot \mathbf{n}_k\|_{L^2(T)} \preceq h \|\mathrm{D}\mathbf{n}_k\|_{L^\infty(T)} \|\mathbf{u}^i_0\|_{L^2(T)}.
\end{equation*}
From this we are able to show that for both the star and point-block patches around $v_i$,
\begin{equation*}
    \begin{aligned}
    \sum_i \| \mathbf{u}_0^i \cdot \mathbf{n}_k\|_{L^2(\mathrm{patch}(v_i))}^2
    &\preceq \sum_i h^2 \|\mathrm{D}\mathbf{n}_k\|_{L^\infty(\mathrm{patch}(v_i))}^2 \|\mathbf{u}^i_0\|_{L^2(\mathrm{patch}(v_i))}^2 \\
    &\preceq h^2 \|\mathrm{D}\mathbf{n}_k\|^2_{L^\infty(\Omega)} \sum_i \|\mathbf{u}_0^i\|^2_{L^2(\Omega)}\\
    &\preceq h^2 \|\mathrm{D}\mathbf{n}_k\|^2_{L^\infty(\Omega)} \| \mathbf{u}_0\|^2_{L^2(\Omega)}.
    \end{aligned}
\end{equation*}
Therefore, with the local support of $\mathbf{u}_0^i$ we have
\begin{equation*}
    \sum_i \| \mathbf{u}_0^i \cdot \mathbf{n}_k\|^2_{L^2(\Omega)}
    = \sum_i \| \mathbf{u}_0^i \cdot \mathbf{n}_k\|_{L^2(\mathrm{patch}(v_i))}^2
    \preceq h^2 \|\mathrm{D}\mathbf{n}_k\|^2_{L^\infty(\Omega)} \| \mathbf{u}_0\|^2_{L^2(\Omega)}.
\end{equation*}
\qed
\end{proof}

We now derive the general form of the spectral bounds in \eqref{eq:spectral}.
This follows a similar approach to \cite[Theorem 4.1]{schoberl-1999-phd-thesis}, but with a different assumption on the splitting approximation, to allow for a dependence on $\gamma$.
For brevity of notation, we respectively denote the standard $L^2$, $H^1$ and $L^{\infty}$ norms by $\|\cdot\|_0$, $\|\cdot\|_1$ and $\|\cdot\|_\infty$. Given a space decomposition $V_h=\sum_i V_i$,
we define its \emph{overlap} $N_O$ as
\begin{equation*}
    N_O \coloneqq \max_{1\le i\le M} \sum_{j=1}^M g_{ij},
\end{equation*}
where
\begin{equation*}
    g_{ij} = \begin{cases}
        1 & \text{if}\ \exists \mathbf{v}_i\in V_i,\mathbf{v}_j\in V_j: |\mathrm{supp}(\mathbf{v}_i)\cap \mathrm{supp}(\mathbf{v}_j)|>0, \\
        0 & \text{otherwise}
    \end{cases}
\end{equation*}
measures the interaction between each subspace.

\begin{theorem}\label{thm:robust-est}
    Let $\{V_i\}$ be a subspace decomposition of $V_h$ with overlap $N_O$.
    Assume that the finite element pair $V_h$-$Q_h$ is inf-sup stable for the mixed problem
    \begin{equation*}
        \begin{aligned}
            B((\mathbf{u},\lambda); (\mathbf{v},\mu))
            &\coloneqq
        K\langle \nabla\mathbf{u},\nabla \mathbf{v}\rangle_0
        + 2 \langle \lambda, \mathbf{n}_k\cdot\mathbf{v}\rangle_0
        + 2\langle \mu, \mathbf{n}_k\cdot\mathbf{u}\rangle_0\\
            &= f(\mathbf{v},\mu)
        \quad \forall (\mathbf{v},\mu)\in V_h\times Q_h,
        \end{aligned}
    \end{equation*}
    where $f$ is a known functional.
Furthermore, assume that the function $\mathbf{u}_h\in V_h$ and the kernel function $\mathbf{u}_0\in {\mathcal{N}}_h$ can be split locally with estimates depending on the mesh size $h$ and possibly on $\gamma$ if the kernel-capturing property is not satisfied:
    \begin{equation*}
        \begin{aligned}
            \underset{\substack{\mathbf{u}_h=\sum_i \mathbf{u}_h^i\\ \mathbf{u}_h^i\in V_i}}{\mathrm{inf}} \sum_i \|\mathbf{u}_h^i\|^2_1 &\le c_1(h)\|\mathbf{u}_h\|_0^2, \\
            \underset{\substack{\mathbf{u}_0=\sum_i \mathbf{u}_0^i\\ \mathbf{u}_0^i\in V_i}}{\mathrm{inf}} \sum_i \|\mathbf{u}_0^i\|^2_{A_{h,\gamma}} &\le \left( c_2(h) + c_3(h,\gamma)\right) \|\mathbf{u}_0\|_0^2.
    \end{aligned}
    \end{equation*}
    Then the additive Schwarz preconditioner $D_{h,\gamma}$ built on the decomposition $\{V_i\}$ satisfies
    \begin{equation}\label{eq:spec-est}
        \left(c_1(h)+c_2(h) + c_3(h,\gamma) \right)^{-1}
        D_{h,\gamma} \le A_{h,\gamma}\le N_O D_{h,\gamma},
    \end{equation}
    with constants $c_1$ and $c_2$ independent of $\gamma$.
\end{theorem}
\begin{proof}
    The upper bound can be directly given by \cite[Lemma 3.2]{schoberl-1999-phd-thesis} independent of the form of partial differential equations.
    
    For the lower bound, choose $\mathbf{u}_h\in V_h$ and split it into $\mathbf{u}_h = \mathbf{u}_0 + \mathbf{u}_1$,
    by solving
    \begin{equation}\label{eq:dec}
        B((\mathbf{u}_1,\lambda_1), (\mathbf{v}_h,\mu_h)) = 2\langle \mu_h, \mathbf{n}_k\cdot\mathbf{u}_h\rangle_0
        \quad \forall (\mathbf{v}_h,\mu_h)\in V_h\times Q_h.
    \end{equation}
    Testing with $\mathbf{v}_h=0$ in \eqref{eq:dec}, we obtain that
    \begin{equation*}
        \langle \mu_h, \mathbf{n}_k\cdot\mathbf{u}_1\rangle_0 =
        \langle \mu_h, \mathbf{n}_k\cdot\mathbf{u}_h\rangle_0 \quad\forall \mu_h\in Q_h.
    \end{equation*}
    Hence, $\mathbf{n}_k\cdot\mathbf{u}_0 = 0$ a.e., that is to say $\mathbf{u}_0\in {\mathcal{N}}_h$.

    By stability of the finite element pair $V_h$-$Q_h$, we have
    \begin{equation*}
        \begin{aligned}
            \|\mathbf{u}_1\|_1
            &\preceq \underset{\substack{\mathbf{v}_h\in V_h\\ \mu_h\in Q_h}}{\mathrm{sup}}
            \frac{B((\mathbf{u}_1,\lambda_1),(\mathbf{v}_h,\mu_h))}{\|(\mathbf{v}_h,\mu_h)\|}\\
            &\preceq \underset{\substack{\mathbf{v}_h\in V_h\\ \mu_h\in Q_h}}{\mathrm{sup}}
            \frac{\|\mathbf{n}_k\cdot\mathbf{u}_h\|_0 \|\mu_h\|_0}{\|(\mathbf{v}_h,\mu_h)\|}\\
            &\le \|\mathbf{n}_k\cdot\mathbf{u}_h\|_0.
        \end{aligned}
    \end{equation*}
    It implies further that
    \begin{equation*}
        \|\mathbf{u}_1\|_1\preceq \|\mathbf{u}_h\|_0
    \end{equation*}
    by the boundedness of $\mathbf{n}_k$ and
    \begin{equation*}
        \|\mathbf{u}_1\|_1\preceq \gamma^{-1/2}\|\mathbf{u}_h\|_{A_{h,\gamma}}
    \end{equation*}
    by the form of the operator $A_{h,\gamma}$, respectively.
    Using $\mathbf{u}_0=\mathbf{u}_h-\mathbf{u}_1$, we have in addition that
    \begin{equation*}
        \|\mathbf{u}_0\|_1\preceq\|\mathbf{u}_h\|_1.
    \end{equation*}

    We now calculate
    \begin{equation}
        \label{eq:cal}
        \begin{aligned}
            \|\mathbf{u}_h\|^2_{D_{h,\gamma}} &= \vertiii{\mathbf{u}_h}^2 \\
            &\le  \underset{\substack{\mathbf{u}_1=\sum_i \mathbf{u}_1^i\\ \mathbf{u}_1^i\in V_i}}{\mathrm{inf}} \sum_i \|\mathbf{u}_1^i \|^2_{A_{h,\gamma}}
            +\underset{\substack{\mathbf{u}_0=\sum_i \mathbf{u}_0^i \\ \mathbf{u}_0^i \in V_i}}{\mathrm{inf}} \sum_i \|\mathbf{u}_0^i\|^2_{A_{h,\gamma}} \\
            &\preceq  (1+\gamma) \underset{\substack{\mathbf{u}_1=\sum_i \mathbf{u}_1^i\\ \mathbf{u}_1^i\in V_i}}{\mathrm{inf}} \sum_i \|\mathbf{u}_1^i \|^2_1
            + \left(c_2(h)+c_3(h,\gamma) \right)  \|\mathbf{u}_0\|_0^2 \\
            &\preceq (1+\gamma)c_1(h)\|\mathbf{u}_1\|_0^2 +
            \left( c_2(h)+ c_3(h,\gamma)\right) \|\mathbf{u}_0\|_1^2 \\
            &\preceq (1+\gamma)c_1(h)\|\mathbf{u}_1\|^2_1 + \left( c_2(h)+c_3(h,\gamma) \right) \|\mathbf{u}_h\|_1^2\\
            &\preceq \left( c_1(h)+ c_2(h) + c_3(h, \gamma) \right) \|\mathbf{u}_h\|_{A_{h,\gamma}}^2,
        \end{aligned}
    \end{equation}
    completing the proof of the spectral estimates \eqref{eq:spec-est}. \qed
\end{proof}

\begin{remark}
    Note that in Theorem \ref{thm:robust-est}, if the kernel-capturing property \eqref{eq:kernel-capture} is satisfied,
    then $c_3$ will be zero.
    Hence, we will instead get a parameter-independent result.
\end{remark}

\begin{corollary}
    \label{cor1}
    In Theorem \ref{thm:robust-est}, if we take $V_h$-$Q_h$ to be constructed by the $[\mathbb{P}_1]^d$-$\mathbb{P}_1$ element, it holds that
    \begin{equation*}
        \left(c_1(h)+c_2(h) + \gamma h^2 \|\mathrm{D}\mathbf{n}_k\|_\infty^2 \right)^{-1}
        D_{h,\gamma} \le A_{h,\gamma}\le N_O D_{h,\gamma},
    \end{equation*}
    with constants $c_1(h)$, $c_2(h)\sim \mathcal{O}(h^{-2})$.
\end{corollary}
\begin{proof}
    We follow the main argument of Theorem \ref{thm:robust-est}. We have only proven the kernel-capturing property for the approximate kernel \eqref{eq:Nhat} rather than \eqref{eq:N*}, and need to account for this in the estimates. From Lemma \ref{lem1} we have that
    \begin{equation*}
        c_3(h,\gamma) = \gamma h^2 \|\mathrm{D}\mathbf{n}_k\|_\infty^2.
    \end{equation*}

    With the choice of $V_h= [\mathbb{P}_1]^d$,
    we will use the so-called \emph{inverse inequality} (its proof can be found in any finite element book, e.g., \cite{ciarlet-fembook}) which states that
    \begin{equation*}
        \|\mathbf{v}_h\|_1 \preceq h^{-1}\|\mathbf{v}_h\|_0 \quad \forall \mathbf{v}_h\in V_h.
    \end{equation*}
    Therefore, it is straightforward to obtain that $c_1$ and $c_2$ are actually $\mathcal{O}(h^{-2})$.
    Notice here we have also used the form of $\|\cdot\|_{A_{h,\gamma}}$ in estimating $c_2(h)$.

    Finally, substituting the form of $c_3$ in \eqref{eq:cal}, we derive
    \begin{equation*}
            \|\mathbf{u}_h\|^2_{D_{h,\gamma}}
            \preceq \left( c_1(h)+ c_2(h) +\gamma h^2 \|\mathrm{D}\mathbf{n}_k\|_\infty^2 \right) \|\mathbf{u}_h\|_{A_{h,\gamma}}^2,
    \end{equation*}
    with constants $c_1(h)$, $c_2(h)\sim \mathcal{O}(h^{-2})$. \qed
\end{proof}

The above Corollary \ref{cor1} implies that we cannot entirely get rid of parameter $\gamma$ in the spectral estimates if the kernel-capturing property for the modified kernel \eqref{eq:N*} is not satisfied and instead we get an additional factor of $\gamma h^2 \|\mathrm{D}\mathbf{n}_k\|_\infty^2$.
However, this $\gamma$-dependence can be well controlled and does not impinge on the effectiveness of our smoother; the dependence improves as the mesh becomes finer or as $\mathbf{n}_k$ becomes smoother.

\subsection{Prolongation}

To construct a parameter-robust multigrid method, the prolongation operator is also required to be continuous (in the energy norm associated with the PDE) with the continuity constant independent of the penalty parameter $\gamma$ \cite[Theorem 4.2]{schoberl-1999-phd-thesis}.
In the context of the Oseen, Navier--Stokes, and linear elasticity equations, the prolongation operator was modified in order to guarantee that the continuity constant is $\gamma$-independent \cite{schoberl-1999-phd-thesis, benzi-2006-article, farrell-mitchell-2018-article}.
However, in our experiments with the Oseen--Frank system, we observe robust convergence with respect to $\gamma$, even when using the (cheaper) standard prolongation (see Section \ref{sec:num-results} for specific details).
This can be seen in Tables \ref{table:almg/star-equal} and \ref{table:almg/pbj-equal} of Section \ref{sec:numerical}, for example.
Hence, we will use the standard prolongation with no modification in this work.

\begin{remark}
    Since both discretizations $[\mathbb{P}_1]^d$-$\mathbb{P}_1$ and $[\mathbb{P}_2]^d$-$\mathbb{P}_1$ are nested, i.e., $V_H\subset V_h$, the standard prolongation $P_H$ is actually a continuous (in the $H^1$-norm) natural inclusion.
\end{remark}

%% file: numerical_results.tex
\section{Numerical experiments}\label{sec:numerical}

\subsection{Algorithm details}

In the following numerical experiments, we use the $[\mathbb{P}_2]^3$-$\mathbb{P}_1$ element pair
and use flexible GMRES \cite{saad-1993-article} as the outermost linear solver,
since GMRES \cite{saad-1986-article} is applied in the multigrid relaxation.
An absolute tolerance of $10^{-8}$ was used for the nonlinear solver, except for the convergence rate tests in Figure \ref{fig:error-pbj}, which used $10^{-10}$. A relative tolerance of $10^{-4}$ was used for the inner linear solver.
We use the full block factorization preconditioner
\begin{equation*}
    P^{-1} =
    \begin{bmatrix}
        I & -\tilde{A}^{-1}_\gamma B^\top \\
        0 & I
    \end{bmatrix}
    \begin{bmatrix}
        \tilde{A}^{-1}_\gamma & 0\\
        0 & \tilde{S}_\gamma^{-1}
    \end{bmatrix}
    \begin{bmatrix}
        I & 0\\
        -B\tilde{A}^{-1}_\gamma & I
    \end{bmatrix},
\end{equation*}
where $\tilde{A}_\gamma^{-1}$ represents solving the top-left block $A_{\gamma}$ inexactly by our specialized multigrid algorithm
and the Schur complement approximation $\tilde{S}_\gamma^{-1}$ is given by \eqref{eq:aug-schur-approx}.
The multiplier mass matrix inverse $M_{\lambda}^{-1}$ is solved using Cholesky factorization.

For $\tilde{A}_{\gamma}^{-1}$, we perform a multigrid V-cycle, where the problem on the coarsest grid is solved exactly by Cholesky decomposition.
On each finer level, as relaxation we perform 3 GMRES iterations preconditioned by
the additive star (denoted as ALMG-STAR) iteration or additive point-block Jacobi (denoted as ALMG-PBJ) iteration.
In order to achieve convergence results independent of the number of cores used in parallel, we only report iteration counts using additive relaxation, although multiplicative ones generally give better convergence.

The solver described above is implemented in the Firedrake \cite{firedrake} library which relies on PETSc \cite{petsc} for solving linear systems.
The star and Vanka relaxation methods are implemented using the PCPATCH preconditioner recently included in PETSc \cite{pcpatch}.

\subsection{Numerical results}
\label{sec:num-results}

All the tests are executed on a computer with an Intel(R) Xeon(R) Silver 4116 CPU@2.10GHz processor.
We denote \#refs and \#dofs as the number of mesh refinements and degrees of freedom, respectively, in the following experiments.

\subsubsection{Periodic boundary condition in a square slab}\label{sec:pbc}

Following the nematic benchmarks in \cite{adler-2016-article}, we consider a generalized twist equilibrium configuration in a square $\Omega = [0,1]\times[0,1]$,
which is proven to have an analytical solution \cite{stewart-2004-book}.
We will investigate the robustness of the solver when applied to unequal Frank constants 
and nonzero cholesteric pitch.

The problem has periodic boundary conditions in the $x$-direction and Dirichlet boundary conditions in the $y$-direction, with values
\begin{displaymath}
    \begin{aligned}
        &\mathbf{n} = [\cos\theta_0, 0, -\sin \theta_0]^\top && \quad \text{on} \quad y=0,\\
        &\mathbf{n} = [\cos\theta_0, 0, \sin \theta_0]^\top && \quad \text{on} \quad y=1,
    \end{aligned}
\end{displaymath}
where $\theta_0={\pi}/8$.

We first consider parameter values $K_1=1.0$, $K_2=1.2$, $K_3=1.0$, $q_0=0$.
The exact solution is given by
\begin{displaymath}
    \mathbf{n} = [\cos(\theta_0(2y-1)), 0, \sin(\theta_0(2y-1))]^\top,
\end{displaymath}
with true free energy $2K_2\theta_0^2 \approx 0.37011$.
An example of the pure twist configuration is illustrated in Figure \ref{fig:rec}.

We use an initial guess of $\mathbf{n}_0=[1,0,0]^\top$ in the Newton iteration and a $10 \times 10$ mesh of triangles of negative slope as the coarse grid.

\begin{figure}[!ht]
    \centering
    \begin{minipage}{0.5\textwidth}
    \centering
        \includegraphics[scale=0.3, trim={7cm 8cm 2cm 8cm}, clip]{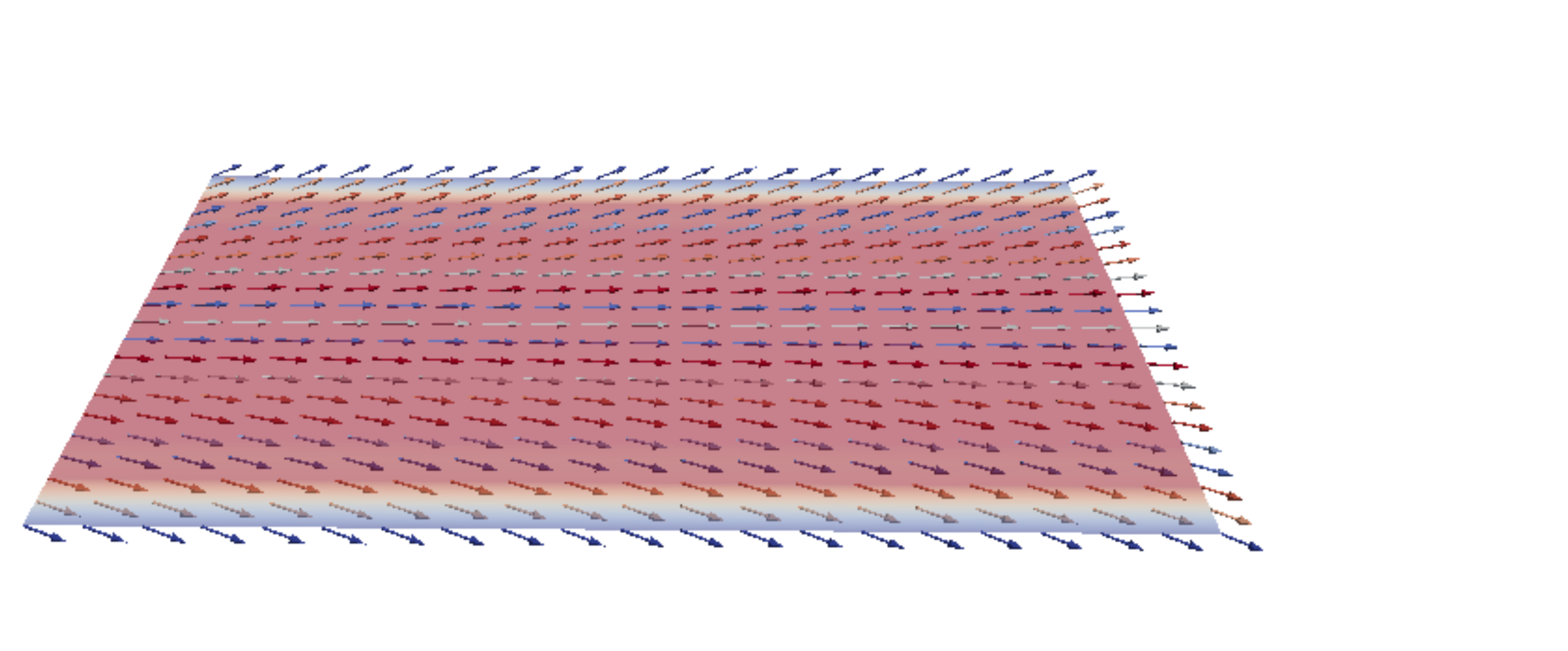}
    \end{minipage}
    \caption{A sample solution of the twist configuration. Colors represent the magnitude of directors.}
    \label{fig:rec}
\end{figure}

We first compare in Table \ref{table:cmp} the nonlinear convergence of the Newton linearization \eqref{eq:newton-aug} against that of the Picard iteration \eqref{eq:modified-newton-aug} we propose.
For these experiments we use the augmented Lagrangian preconditioner with ideal inner solvers (denoted as ALLU), i.e.~where the top-left block is solved exactly by LU factorization.
The Picard iteration requires substantially fewer nonlinear iterations for large $\gamma$. We expect that this relates to the degradation of the coercivity estimate given in Lemma \ref{lem:acoer-aug}, and will be analyzed in future work. Similar results were obtained on other test cases and we adopt the Picard iteration henceforth.

\begin{table}[!ht]
\centering
    \begin{tabular}{p{0.2\textwidth}crllll}
        \toprule
          & &  & \multicolumn{4}{c}{$\gamma$} \\ \cmidrule(lr){4-7}
          & \#refs & \#dofs & $10^3$ & $10^4$ & $10^5$ & $10^6$ \\
         \midrule
         \multirow{4}{*}{Newton}
          & 1 & 5,340 & 2.20 (5) & 1.14 (7) & 1.00 (10) & 1.00 (19) \\
         &2 & 21,080 & 3.20 (5) & 1.14 (7) & 1.00 (12) & 1.00 (15) \\
         &3 & 83,760 & 3.83 (6) & 1.57 (7) & 1.11 (9) & 1.00 (14) \\
         &4 & 333,920 & 4.67 (6) & 2.14 (7) & 1.00 (7) & 1.00 (11) \\
        &5 & 1,333,440 & 5.17 (6) & 2.43 (7) & 1.57 (7) & 1.00 (10)\\
        \bottomrule
         \multirow{4}{*}{Picard}
         &1 & 5,340 & 2.00 (5) & 1.20 (5) & 1.14 (7) & 1.11 (9)\\
         &2 & 21,080 & 3.00 (5) & 1.40 (5) & 1.17 (6) & 1.12 (8)\\
        &3 & 83,760 & 3.83 (6) & 2.00 (5) & 1.17 (6) & 1.14 (7)\\
        &4 & 333,920 & 4.67 (6) & 2.29 (7) & 1.14 (7) & 1.17 (6)\\
        &5 & 1,333,440 & 5.17 (6) & 2.57 (7) & 1.50 (8) & 1.17 (6)\\
        \bottomrule
    \end{tabular}
    \caption{A comparison of the nonlinear convergence of the Newton linearization \eqref{eq:newton-aug} and the Picard iteration \eqref{eq:modified-newton-aug}
    using ideal inner solvers for a nematic LC problem in a square slab.
    The table shows the average number of outer FGMRES iterations per nonlinear iteration and the total nonlinear iterations in brackets.}
    \label{table:cmp}
\end{table}

To see the efficiency of the Schur complement approximation \eqref{eq:aug-schur-approx} we used in Section \ref{sec:schur},
we give the number of Krylov iterations for ALLU in Table \ref{table:allu-rec}.
It can be observed that as $\gamma$ increases, the preconditioner becomes a better approximation to the real Jacobian inverse
and that the preconditioner is mesh-independent.

\begin{table}[!ht]
\centering
    \begin{tabular}{crllllllll}
        \toprule
         &  & \multicolumn{8}{c}{$\gamma$} \\ \cmidrule(lr){3-10}
        \#refs & \#dofs & 0 & 1 & 10 & $10^2$ & $10^3$ & $10^4$ & $10^5$ & $10^6$ \\
         \midrule
        1 & 5,340 & 10.40 & 9.20 & 8.00 & 5.40 & 2.00 & 1.20 & 1.14 & 1.11 \\
        2 & 21,080 & 14.20 & 13.20 & 9.20 & 5.80 & 3.00 & 1.40 & 1.17 & 1.12 \\
        3 & 83,760 & 4.75 & 4.75 & 6.75 & 6.40 & 3.83 & 2.00 & 1.17 & 1.14 \\
        4 & 333,920 & 5.50 & 4.50 & 7.25 & 7.20 & 4.67 & 2.29 & 1.14 & 1.17 \\
        5 & 1,333,440 & 5.25 & 3.75 & 5.75 & 7.00 & 5.17 & 2.57 & 1.50 & 1.17 \\
        \bottomrule
    \end{tabular}
    \caption{ALLU: The average number of FGMRES iterations per Newton iteration for a nematic LC problem
    in a square slab using $[\mathbb{P}_2]^3$-$\mathbb{P}_1$ discretization.}
    \label{table:allu-rec}
\end{table}

The performance of ALMG-STAR and ALMG-PBJ are illustrated in Tables \ref{table:almg/star} and \ref{table:almg/pbj}, respectively,
where both mesh-independence for $\gamma=10^6$ and $\gamma$-robustness are observed.

\begin{table}[!ht]
\centering
    \begin{tabular}{crllll}
        \toprule
         &  & \multicolumn{4}{c}{$\gamma$} \\ \cmidrule(lr){3-6}
         \#refs & \#dofs & $10^3$ & $10^4$ & $10^5$ & $10^6$\\
         \midrule
         1 & 5,340 & 2.60 (5) & 2.40 (5) & 2.29 (7) & 2.29 (7)\\
         2 & 21,080 & 4.20 (5) & 2.20 (5) & 2.50 (6) & 3.29 (7)\\
         3 & 83,760 & 8.00 (5) & 3.00 (5) & 2.33 (6) & 3.33 (6)\\
         4 & 333,920 & 11.60 (5) & 5.17 (6) & 2.17 (6) & 2.29 (7)\\
         5 & 1,333,440 & 15.20 (5) & 8.43 (7) & 3.14 (7) & 1.78 (9)\\
        \bottomrule
    \end{tabular}
    \caption{ALMG-STAR: the average number of FGMRES iterations per Newton iteration (total Newton iterations) for the nematic LC problem in a square slab.}
\label{table:almg/star}
\end{table}

\begin{table}[!ht]
\centering
\begin{tabular}{crp{1.35cm}lll}
        \toprule
        & & \multicolumn{4}{c}{$\gamma$} \\ \cmidrule(lr){3-6}
        \#refs & \#dofs & $10^3$ & $10^4$ & $10^5$ & $10^6$\\
        \midrule
         1 & 5,340 & 3.20 (5) & 2.60 (5) & 3.00 (6) & 3.57 (7)\\
         2 & 21,080 & 5.60 (5) & 2.60 (5) & 2.83 (6) & 3.71 (7)\\
         3 & 83,760 & 10.00 (5) & 3.80 (5) & 2.80 (5) & 3.00 (6)\\
        4 & 333,920 & 15.40 (5) & 7.00 (5) & 2.50 (6) & 2.83 (6)\\
        5 & 1,333,440 & $>$100 & 11.83 (6) & 5.00 (5) & 2.83 (6)\\
        \bottomrule
    \end{tabular}
    \caption{ALMG-PBJ: the average number of FGMRES iterations per Newton iteration (total Newton iterations) for the nematic LC problem in a square slab.}
\label{table:almg/pbj}
\end{table}

We also test the robustness of ALMG-STAR and ALMG-PBJ on other problem parameters, the twist elastic constant $K_2>0$ and the cholesteric pitch $q_0$.
To this end, we continue $K_2\in [0.2, 8]$ and $q_0\in [0,8]$ with step $0.1$. We fix $\gamma=10^6$, since it gives the best performance in Tables \ref{table:almg/star} and \ref{table:almg/pbj}.
The numerical results of ALMG-STAR and ALMG-PBJ in $K_2$- and $q_0$-continuation are shown in Figures \ref{fig:k2} and \ref{fig:q0}, respectively.
Clearly, a stable number of linear iterations is shown for both continuation experiments.

\begin{figure}[!ht]
    \centering
    \begin{minipage}{0.5\textwidth}
        \centering
        \includegraphics[width=1\textwidth]{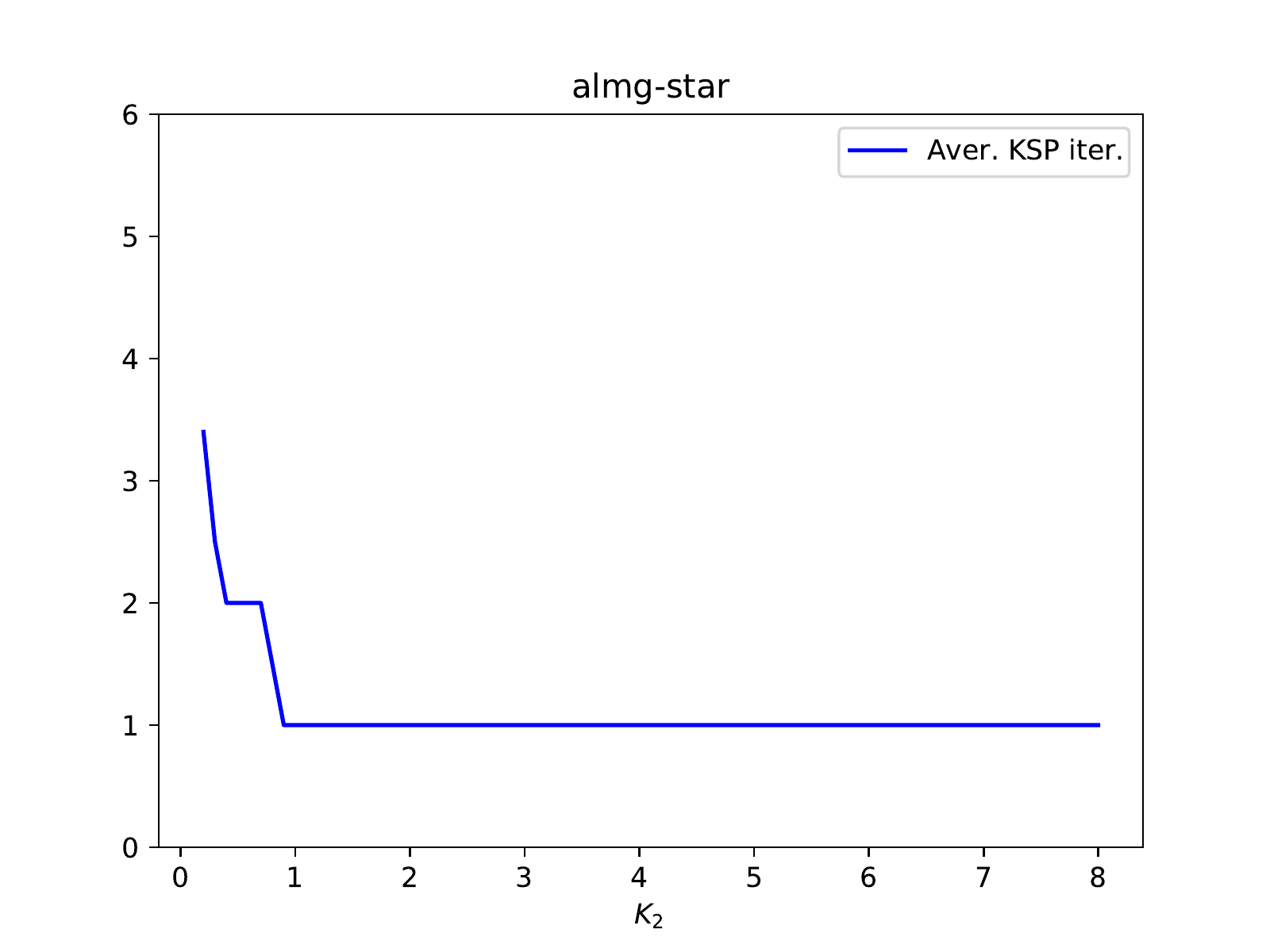}
    \end{minipage}\hfill
    \begin{minipage}{0.5\textwidth}
        \centering
        \includegraphics[width=1\textwidth]{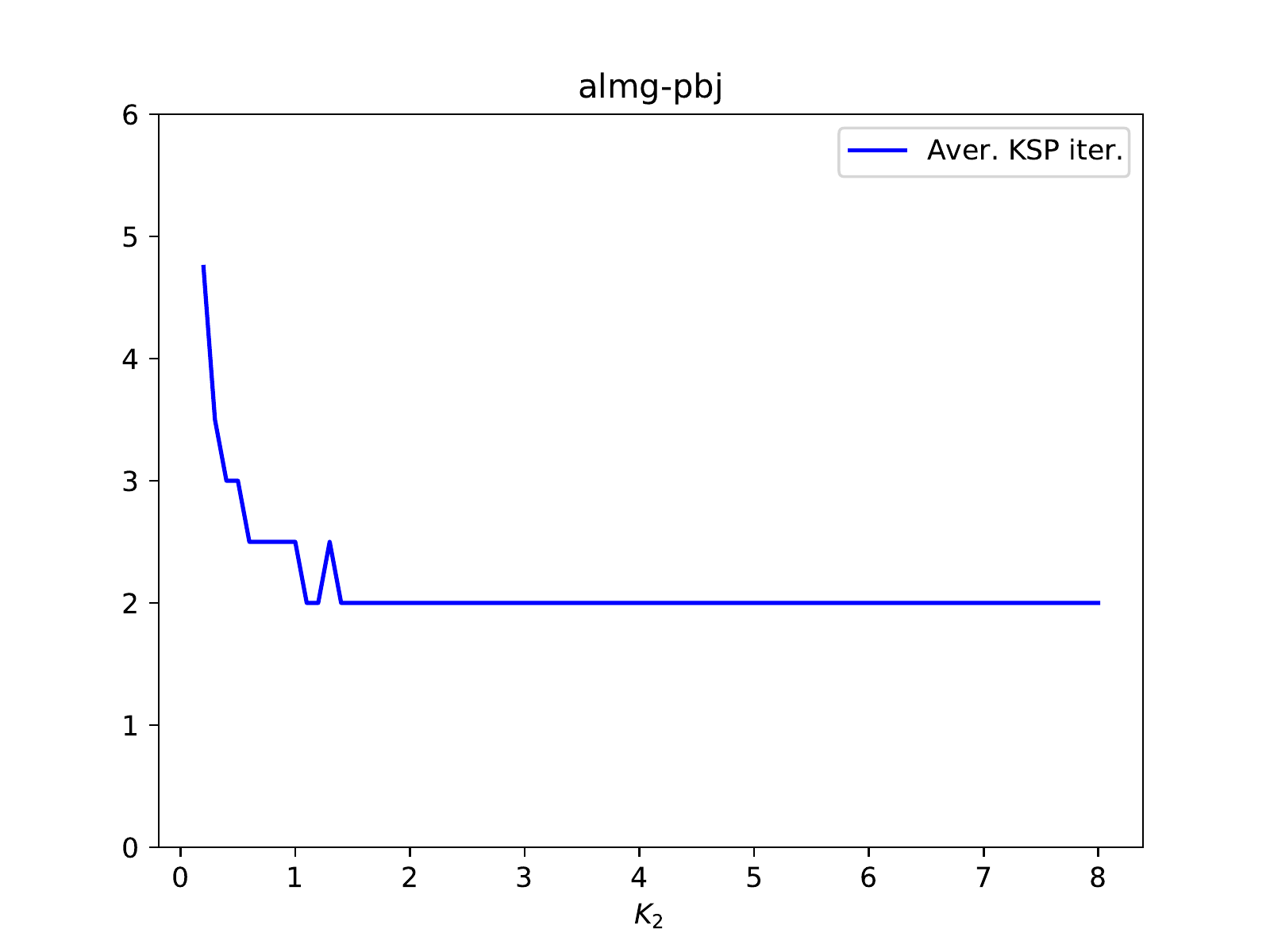}
    \end{minipage}\hfill
    \caption{Average number of FGMRES iterations per Newton iteration when continuing in $K_2$ for the LC problem in a square slab.}
    \label{fig:k2}
\end{figure}

\begin{figure}[!ht]
    \centering
    \begin{minipage}{0.5\textwidth}
        \centering
        \includegraphics[width=1\textwidth]{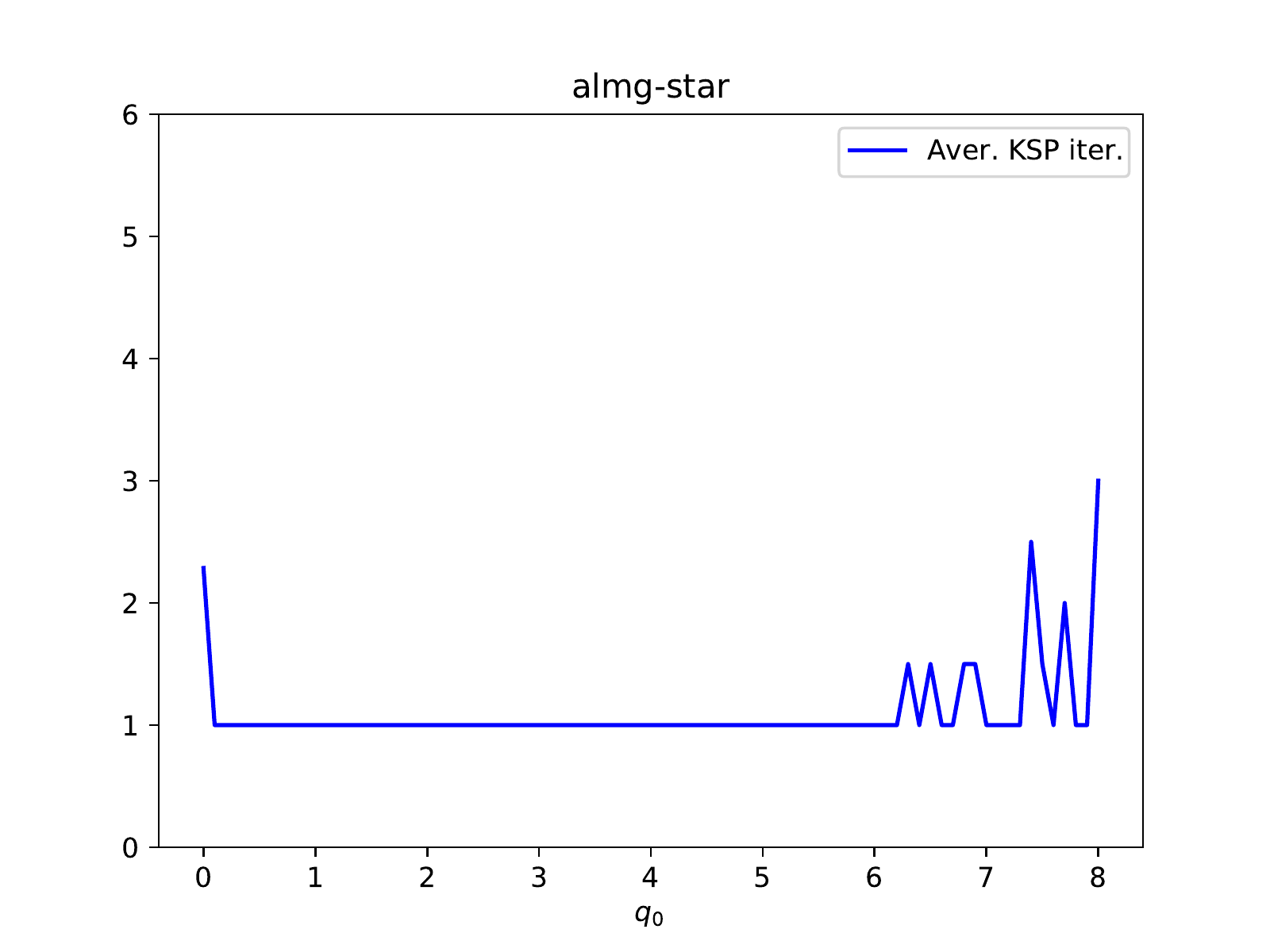}
    \end{minipage}\hfill
    \begin{minipage}{0.5\textwidth}
        \centering
        \includegraphics[width=1\textwidth]{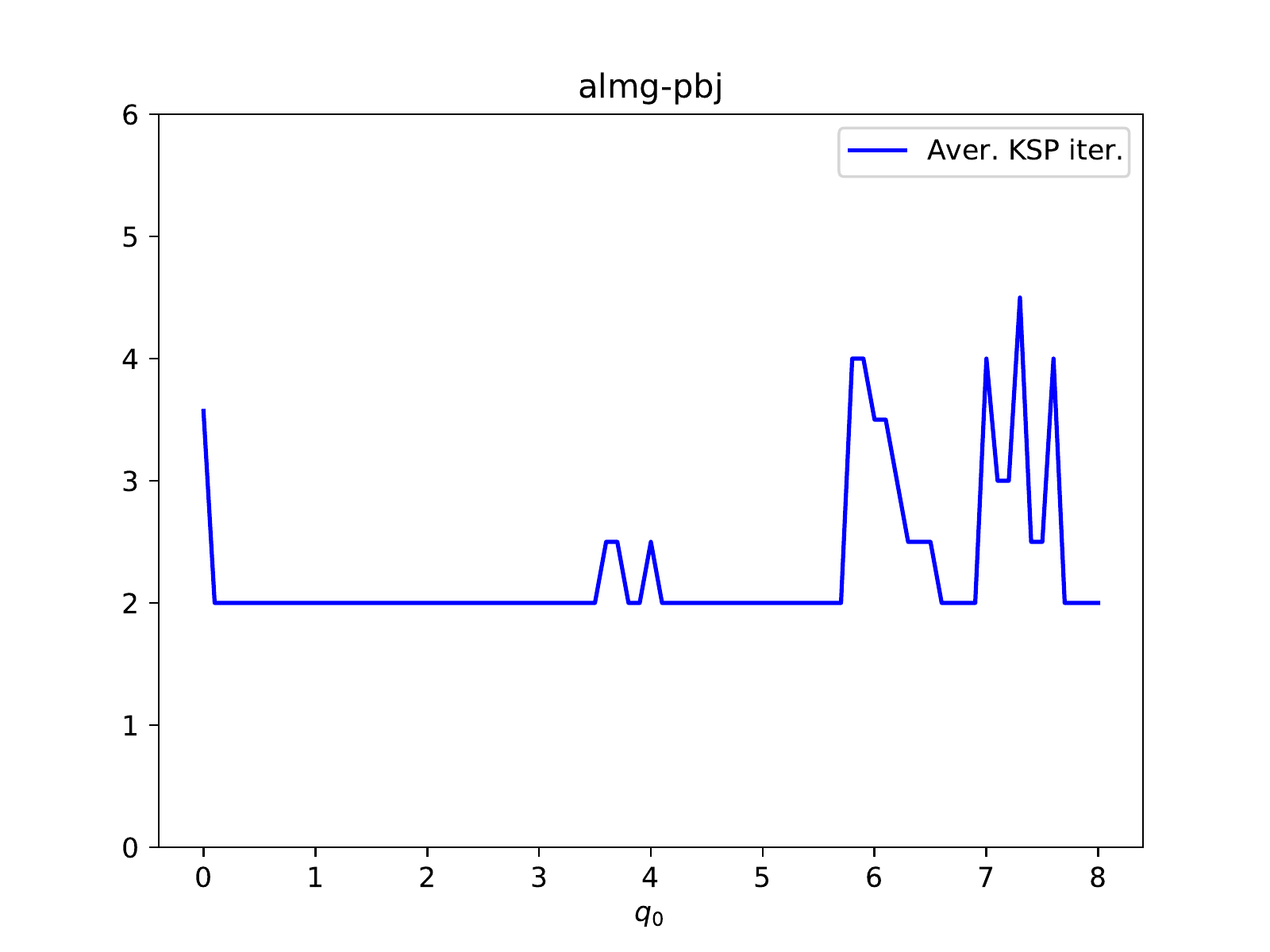}
    \end{minipage}\hfill
    \caption{Average number of FGMRES iterations per Newton iteration when continuing in $q_0$ for the LC problem in a square slab.}
    \label{fig:q0}
\end{figure}

To examine the convergence order of the discretization as a function of $\gamma$, we apply the ALMG-PBJ solver for $\gamma=10^4, 10^5$ and $10^6$.
Note that the convergence result does not rely on the solver used.
Figure \ref{fig:error-pbj} shows the $L^2$- and $H^1$-error between the computed director and the known analytical solution.
We observe third order convergence of the director in the $L^2$ norm and second order convergence in the $H^1$ norm for all values of $\gamma$ considered.
\begin{figure}[!ht]
    \centering
    \begin{minipage}{0.5\textwidth}
        \centering
        \includegraphics[width=1\textwidth]{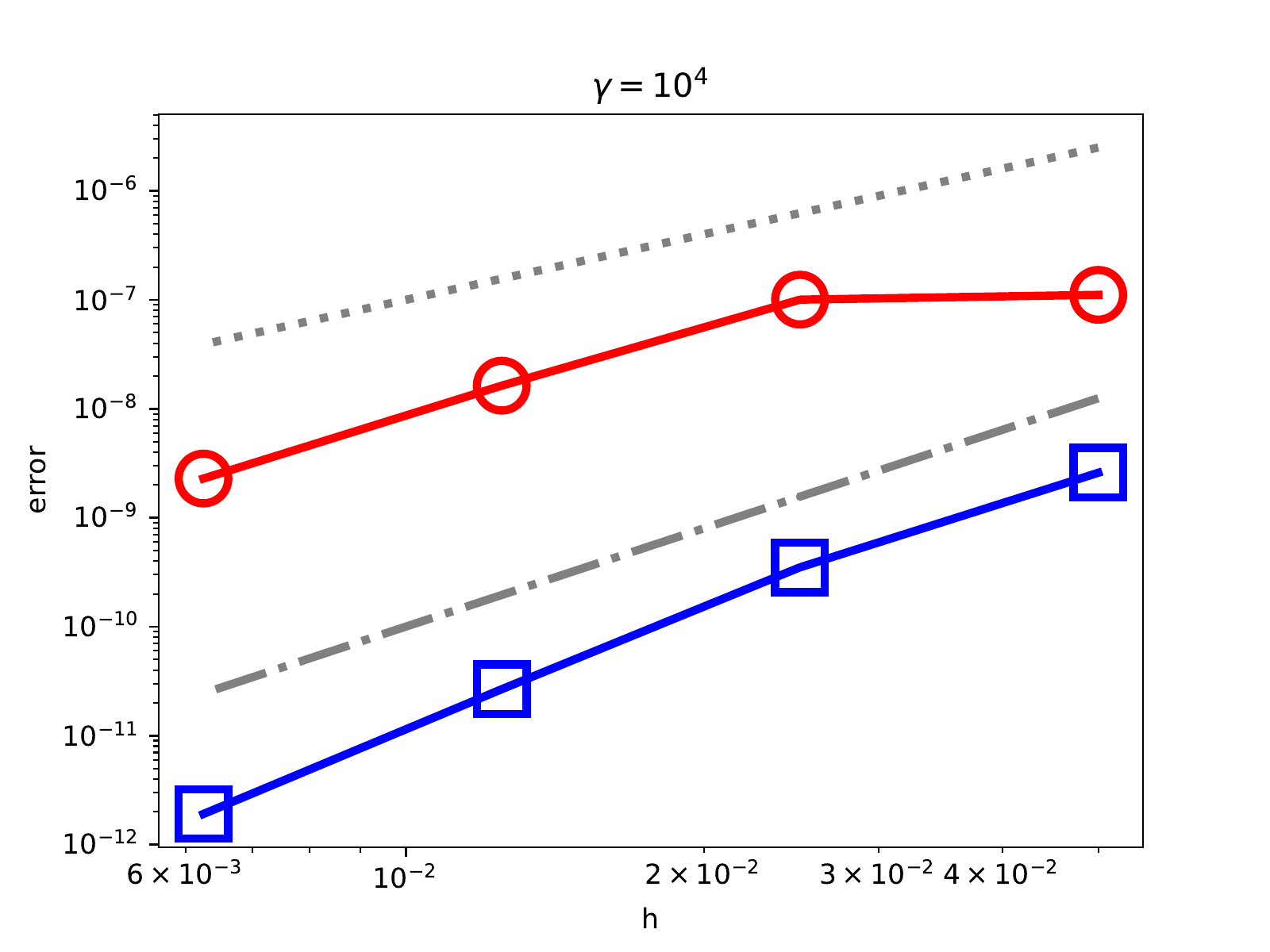}\\
    \end{minipage}\hfill
    \begin{minipage}{0.5\textwidth}
        \centering
        \includegraphics[width=1\textwidth]{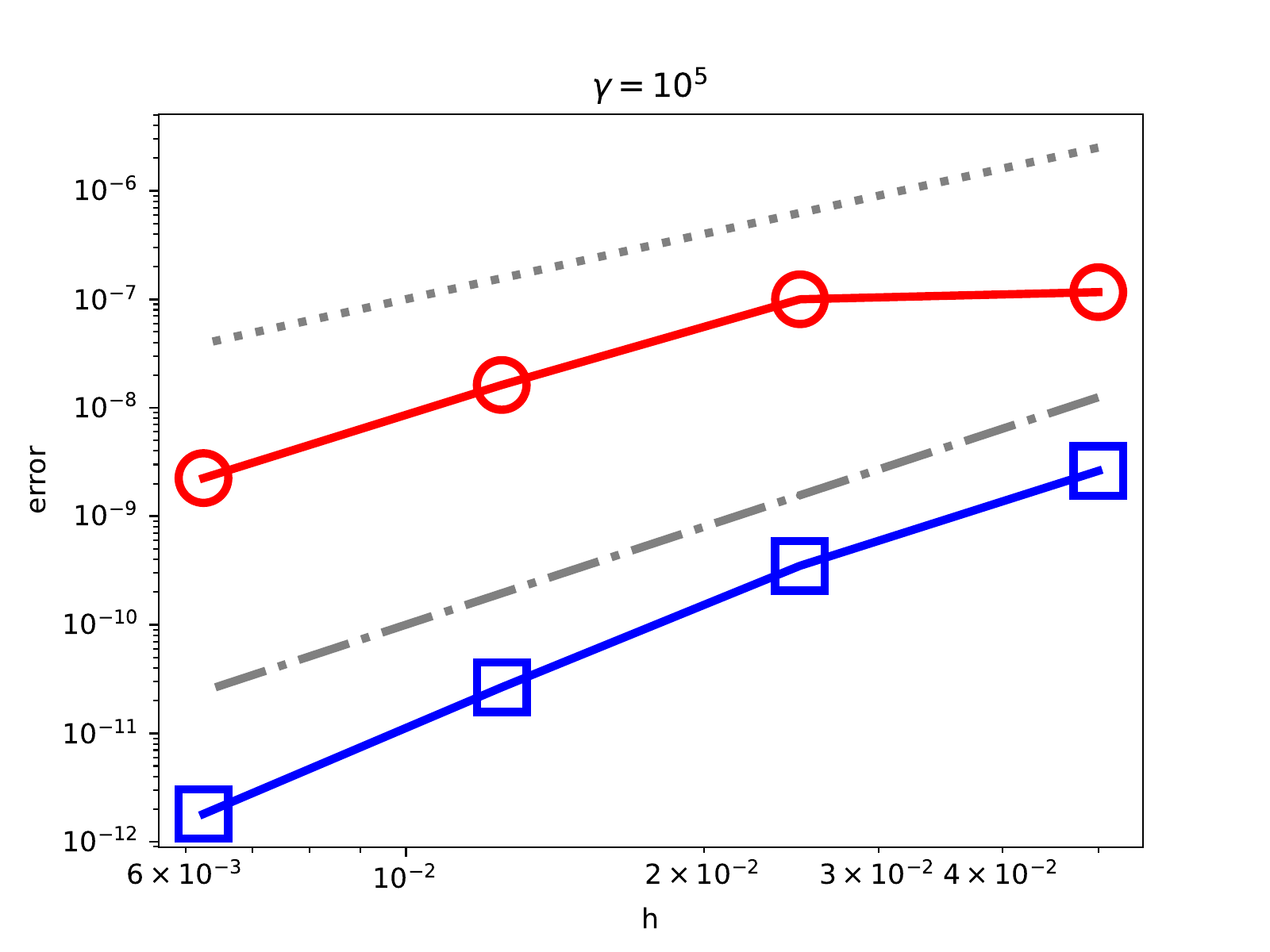}\\
    \end{minipage}\hfill
    \begin{minipage}{0.5\textwidth}
        \centering
        \includegraphics[width=1\textwidth]{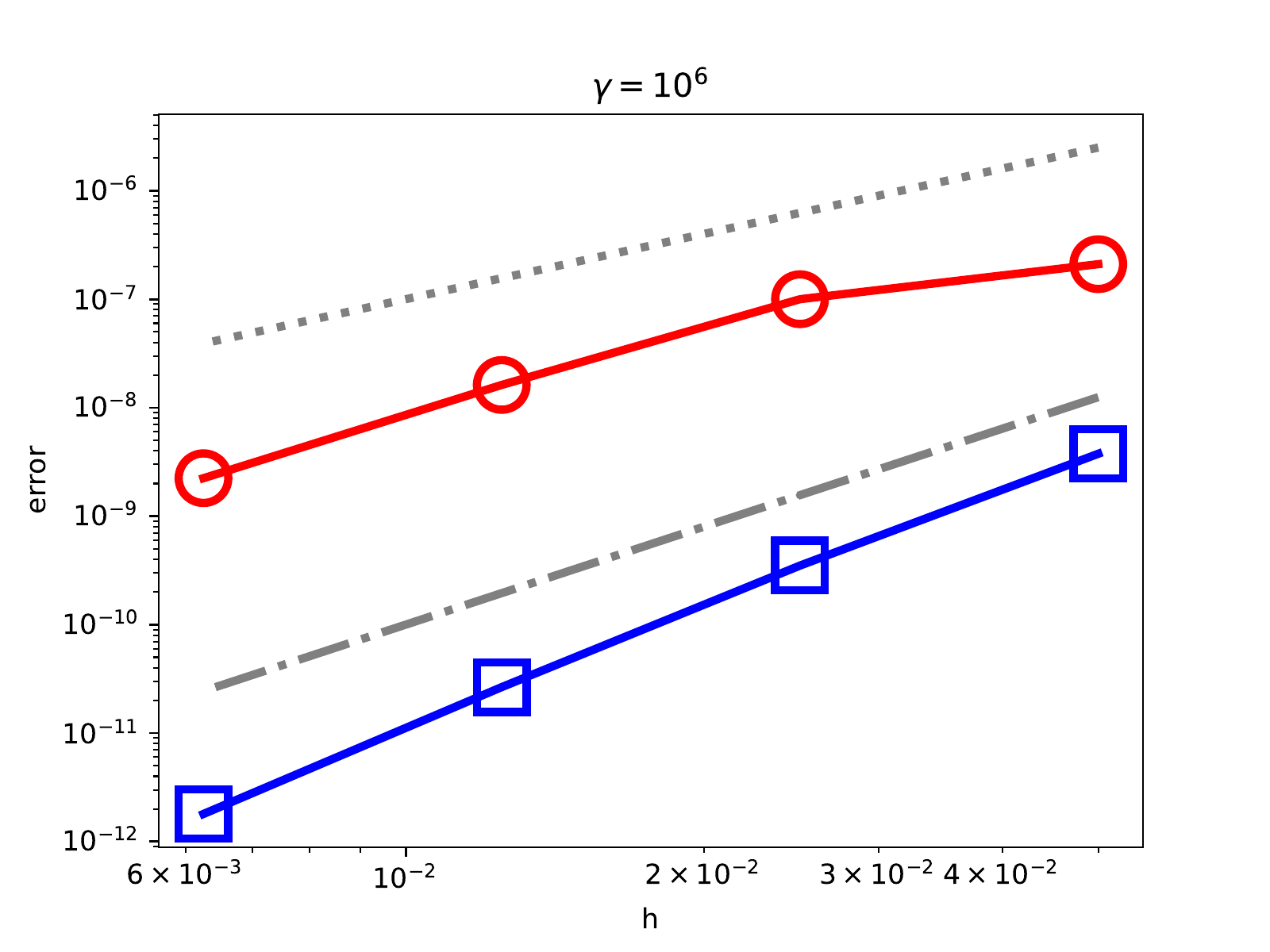}
    \end{minipage}\hfill
    \begin{minipage}{0.5\textwidth}
        \centering
        \includegraphics[width=1\textwidth, scale=1, trim={5cm 4cm 5cm 4cm}, clip]{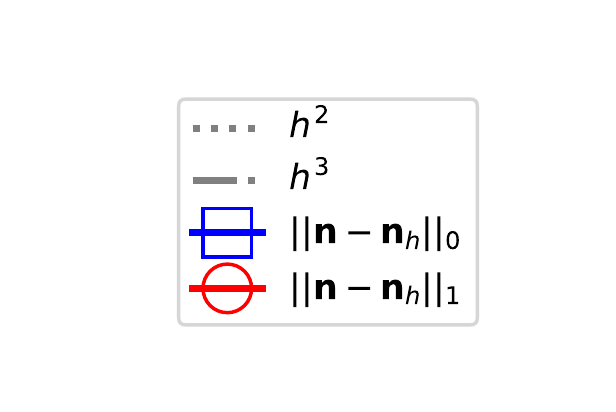}
    \end{minipage}\hfill
    \caption{The convergence of the computed director as the mesh is refined for the nematic LC problem in a square slab.}
    \label{fig:error-pbj}
\end{figure}

To investigate the computational efficiency of the AL approach,
we compare our proposed AL-based solvers (ALMG-PBJ and ALMG-STAR) with a monolithic multigrid preconditioner using Vanka relaxation \cite{adler-2015b-article, vanka-1986-article} on each level (denoted as MGVANKA) in Table \ref{table:runtime}.
Essentially, MGVANKA applies multigrid to the coupled director-multiplier problem, with an additive Schwarz relaxation organised around gathering all director dofs coupled to a given multiplier dof.
All results are computed in serial.
In our experiments, these two AL-based solvers outperform MGVANKA even for small problems of about five thousand dofs.
In particular, ALMG-PBJ is the fastest method considered and is approximately five times faster than MGVANKA for a problem with about five million dofs.
We also notice that ALMG-STAR is slower than ALMG-PBJ, which is caused by the size of the star patch being larger than that of the point-block patch, requiring more work in the multigrid relaxation.

\begin{table}[!ht]
\centering
    \begin{tabular}{lllllll}
    \toprule
    \multicolumn{7}{c}{Computing time (in minutes)}\\
        \midrule
        \#refs & 1 & 2 & 3 & 4 & 5 & 6\\
        \#dofs & 5,340 & 21,080 & 83,760 & 333,920 & 1,333,440 & 5,329,280\\
        ALMG-PBJ & 0.02 & 0.04 & 0.09 & 0.32 & 1.17 & 5.53\\
        ALMG-STAR & 0.02 & 0.07 & 0.23 & 0.79 & 2.95 & 12.86\\
        MGVANKA & 0.04 & 0.15 & 0.38 & 1.44 & 5.91 & 25.09\\
        \bottomrule
    \end{tabular}
    \caption{The computing time of ALMG-PBJ, ALMG-STAR and MGVANKA as a function of mesh refinement for the nematic LC problem in a square slab.}
\label{table:runtime}
\end{table}

\subsubsection{Equal-constant nematic case in an ellipse}

Consider an ellipse of aspect ratio ${3}/{2}$ with strong anchoring boundary condition $\mathbf{n} = [0, 0, 1]^\top$ imposed on the entire boundary.
We consider the equal-constant nematic case $K_1=K_2=K_3=1$, $q_0=0$ to verify the theoretical results presented in previous sections with corresponding discretizations.
We use the initial guess $\mathbf{n}_0 = [0, 0, 0.8]^\top$ in the nonlinear iteration.
The coarsest triangulation, generated in Gmsh \cite{gmsh}, is illustrated in Figure \ref{fig:ellipse}.

\begin{figure}[!ht]
    \centering
    \includegraphics[width=0.5\textwidth, trim={5cm 6cm 5cm 6cm}, clip]{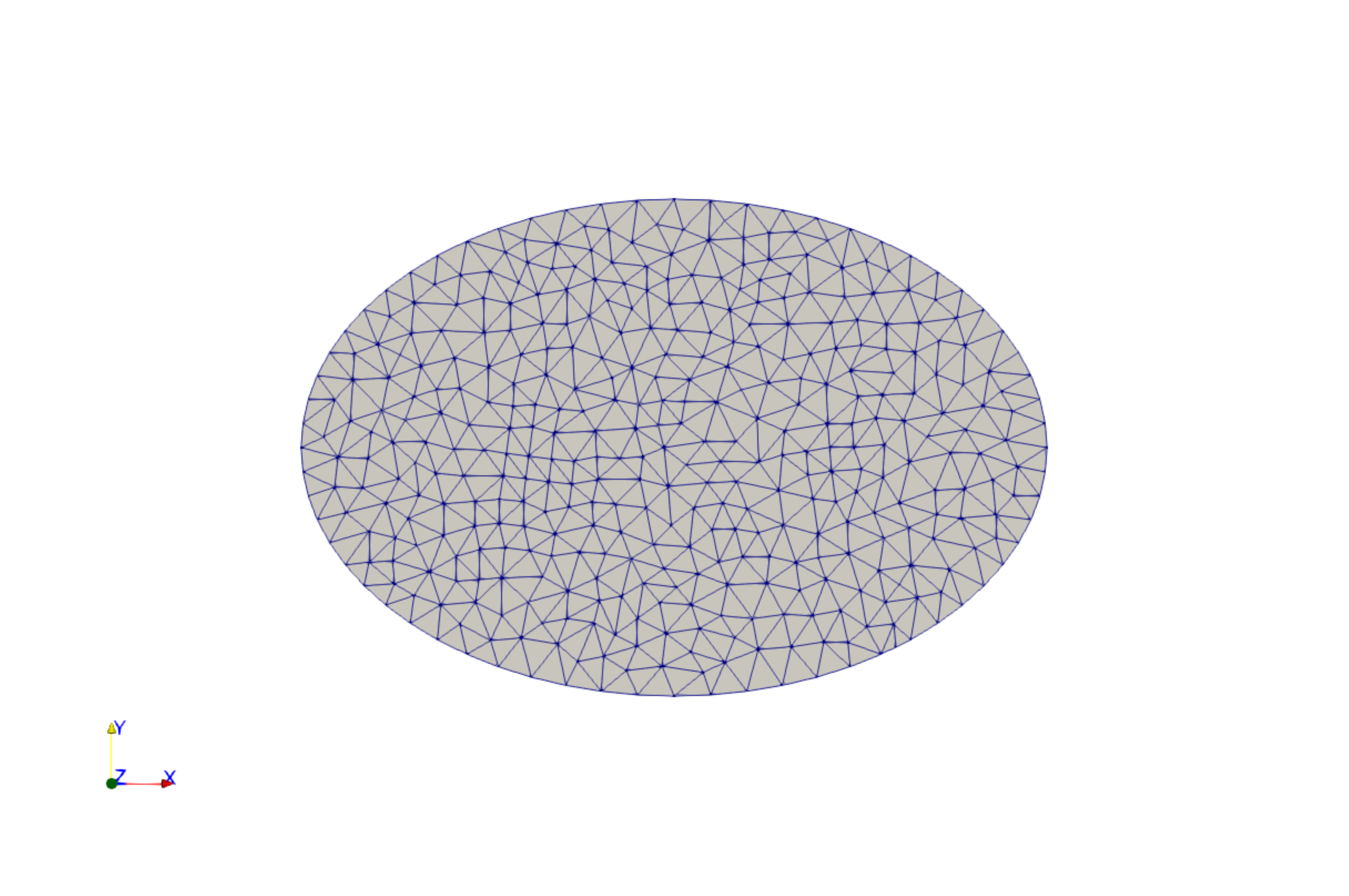}
    \caption{The coarse mesh of the ellipse.}
    \label{fig:ellipse}
\end{figure}

To verify our theoretical results on the improvement of the discrete enforcement of the constraint in Section \ref{sec:improv},
we vary the penalty parameter $\gamma$, use one refinement for the fine mesh, and employ the $[\mathbb{P}_1]^3$-$\mathbb{P}_1$ element.
The data is plotted in Figure \ref{fig:improv}.
The $L^2$-norm $\|\mathbf{n}\cdot\mathbf{n}-1\|_0$ of the residual of the constraint decreases as $\gamma$ grows,
and scales like $\mathcal{O}(\gamma^{-1/2})$ as expected.

\begin{figure}[!ht]
    \centering
    \includegraphics[width=0.5\textwidth]{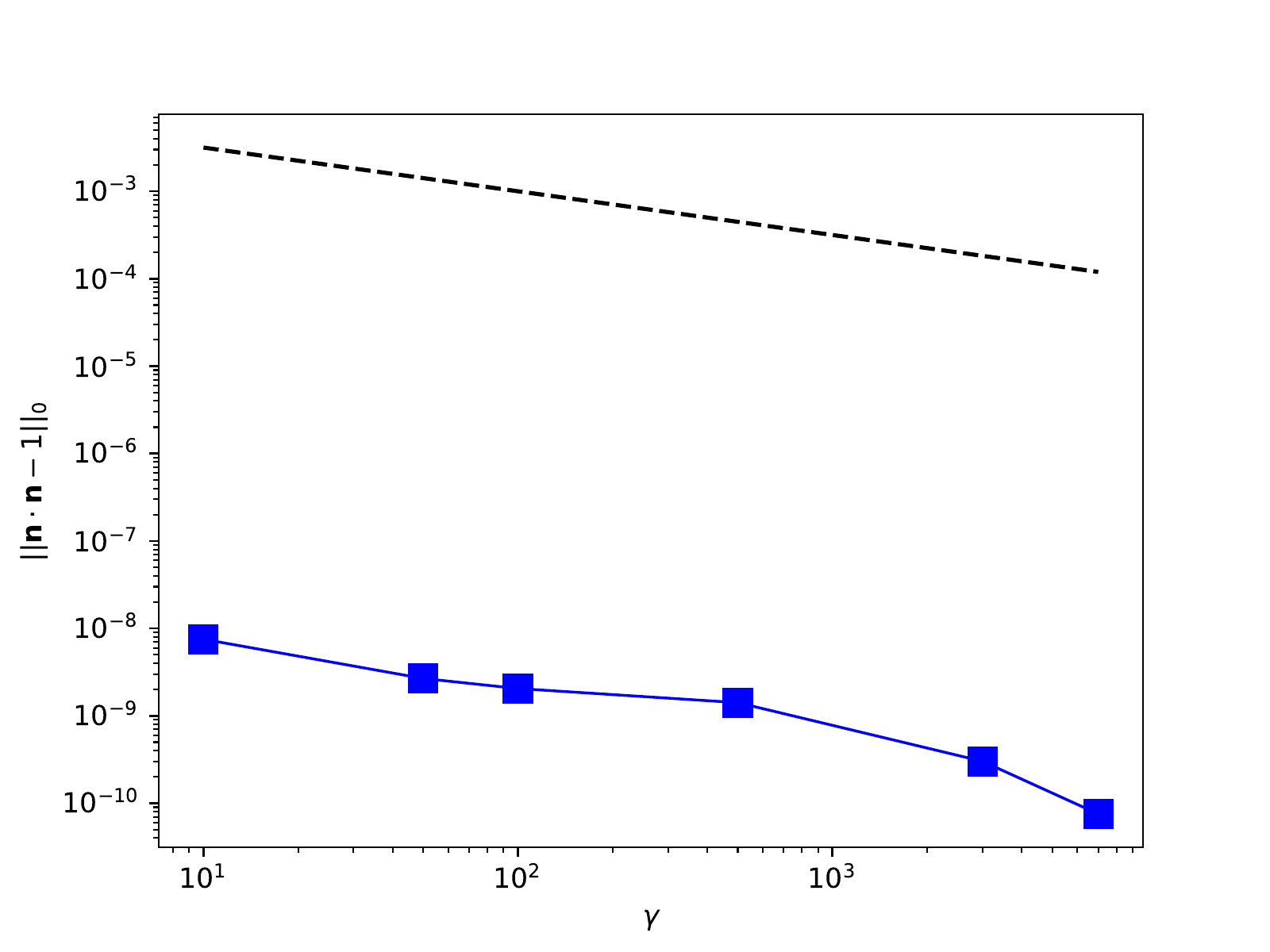}
    \caption{Comparison of the computed constraint $\|\mathbf{n}\cdot\mathbf{n}-1\|_0$ and the reference line $\mathcal{O}(\gamma^{-1/2})$
    using the $[\mathbb{P}_1]^3$-$\mathbb{P}_1$ finite element pair for equal-constant nematic LC problems in an ellipse.}
    \label{fig:improv}
\end{figure}

The efficiency of the Schur complement approximation of Section \ref{sec:schur} for the $[\mathbb{P}_2]^3$-$\mathbb{P}_1$ element can be observed in Table \ref{table:allu-ell}.

\begin{table}[!ht]
\centering
    \begin{tabular}{crllllllll}
        \toprule
         &  & \multicolumn{8}{c}{$\gamma$} \\ \cmidrule(lr){3-10}
        \#refs & \#dofs & 0 & 1 & 10 & $10^2$ & $10^3$ & $10^4$ & $10^5$ & $10^6$ \\
         \midrule
        1 & 19,933 & 29.20 & 25.60 & 16.40 & 5.20 & 2.60 & 1.60 & 1.33 & 1.14 \\
        2 & 78,810 & 32.50 & 26.00 & 14.00 & 6.80 & 3.40 & 1.80 & 1.33 & 1.17 \\
        3 & 313,408 & 12.50 & 15.50 & 16.25 & 7.60 & 4.20 & 2.20 & 1.33 & 1.17 \\
        4 & 1,249,980 & 11.00 & 12.25 & 14.75 & 8.40 & 4.80 & 2.60 & 1.40 & 1.17 \\
        5 & 4,992,628 & 12.33 & 13.33 & 11.75 & 8.00 & 5.20 & 3.00 & 1.50 & 1.14 \\
        \bottomrule
    \end{tabular}
    \caption{ALLU: The average number of FGMRES iterations per Newton iteration for an equal-constant nematic problem in an ellipse using $[\mathbb{P}_2]^3$-$\mathbb{P}_1$ discretization.}
    \label{table:allu-ell}
\end{table}

Tables \ref{table:almg/star-equal} and \ref{table:almg/pbj-equal} demonstrate
the robustness of ALMG-STAR and ALMG-PBJ with respect to $\gamma$ and mesh refinement
for the $[\mathbb{P}_2]^3$-$\mathbb{P}_1$ element.
It can be seen that both solvers are robust with respect to the penalty parameter $\gamma$,
and with respect to the mesh size $h$ for $\gamma=10^6$.
The number of nonlinear iterations and the number of FGMRES iterations per Newton step remain
stable.

\begin{table}[!ht]
\centering
    \begin{tabular}{crllll}
        \toprule
         &  & \multicolumn{4}{c}{$\gamma$} \\ \cmidrule(lr){3-6}
         \#refs & \#dofs & $10^3$ & $10^4$ & $10^5$ & $10^6$\\
         \midrule
         1 & 19,933 & 2.60 (5) & 1.60 (5) & 1.80 (5) & 1.67 (6)\\
         2 & 78,810 & 4.40 (5) & 1.80 (5) & 1.60 (5) & 1.50 (6)\\
         3 & 313,408 & 6.80 (5) & 3.20 (5) & 1.50 (6) & 1.50 (6)\\
        4 & 1,249,980 & 10.00 (5) & 4.67 (6) & 1.80 (5) & 1.50 (6)\\
        5 & 4,992,628 & 14.40 (5) & 7.50 (6) & 4.20 (5) & 1.33 (6)\\
        \bottomrule
    \end{tabular}
    \caption{ALMG-STAR: the average number of FGMRES iterations per Newton iteration (total Newton iterations) for equal-constant nematic problem in an ellipse using $[\mathbb{P}_2]^3$-$\mathbb{P}_1$ discretization.}
\label{table:almg/star-equal}
\end{table}

\begin{table}[!ht]
\centering
    \begin{tabular}{crllll}
        \toprule
         &  & \multicolumn{4}{c}{$\gamma$} \\ \cmidrule(lr){3-6}
         \#refs & \#dofs & $10^3$ & $10^4$ & $10^5$ & $10^6$\\
         \midrule
         1 &19,933 & 3.80 (5) & 2.60 (5) & 2.60 (5) & 2.80 (5)\\
         2 & 78,810 & 6.80 (5) & 3.20 (5) & 2.60 (5) & 2.60 (5)\\
         3 & 313,408 & 9.00 (5) & 5.00 (5) & 2.60 (5) & 2.60 (5)\\
        4 & 1,249,980 & 14.80 (5) & 8.20 (5) & 3.80 (5) & 2.40 (5)\\
        5 & 4,992,628 & 19.00 (5) & 11.60 (5) & 6.80 (5) & 2.50 (6)\\
        \bottomrule
    \end{tabular}
    \caption{ALMG-PBJ: the average number of FGMRES iterations per Newton iteration (total Newton iterations) for equal-constant nematic problem in an ellipse using $[\mathbb{P}_2]^3$-$\mathbb{P}_1$ discretization.}
\label{table:almg/pbj-equal}
\end{table}

\textbf{Code availability.} For reproducibility, both the solver code \cite{zenodo-alpaper} and the exact version of Firedrake used \cite{zenodo-firedrake-20201106} to produce the numerical results of this paper have been archived on Zenodo. An installation of Firedrake with components matching those used in this paper can be obtained by following the instructions at \url{https://www.firedrakeproject.org/download.html} with
\begin{verbatim}
    python3 firedrake-install --doi 10.5281/zenodo.4249051
\end{verbatim}